\documentclass[a4paper,11pt]{article}
\usepackage[T1]{fontenc}
\usepackage{lmodern}
\usepackage[utf8]{inputenc}
\usepackage[USenglish]{babel}
\usepackage[babel=true]{microtype}
\usepackage{amsmath}
\usepackage{amsthm}
\usepackage{mathtools}
\usepackage{amssymb}
\usepackage{commath}
\usepackage{mathrsfs}
\usepackage[pdftex]{graphicx}
\usepackage{pgfplots}
\usepackage{tikz}
\usepackage{fancyhdr}
\usepackage{thm-restate}
\usepackage[disable,colorinlistoftodos,color=yellow,textsize=small]{todonotes}
\usepackage{float}
\usepackage[numbers, sort]{natbib}
\usepackage{hyperref}
\usepackage[labelfont=bf]{caption}
\usepackage{subcaption}
\usepackage[capitalise]{cleveref}
\usepackage[all]{hypcap} 
\usepackage{enumitem}
\usepackage{bm} 
\usepackage[normalem]{ulem}
\usepackage{typearea}

\hypersetup{%
pdftitle=,
pdfauthor=Max Reppen,
pdfcreator=,
pdfproducer=,
colorlinks=true,
linkcolor=black,
citecolor=black,
filecolor=black,
urlcolor=black,
bookmarksopen,
bookmarksopenlevel=1,
}

\pgfplotsset{compat=1.14}
\pgfplotsset{
colormap/viridis/.style={
	colormap={viridis}{
	  rgb=(0.26700401,  0.00487433,  0.32941519)
	  rgb=(0.26851048,  0.00960483,  0.33542652)
	  rgb=(0.26994384,  0.01462494,  0.34137895)
	  rgb=(0.27130489,  0.01994186,  0.34726862)
	  rgb=(0.27259384,  0.02556309,  0.35309303)
	  rgb=(0.27380934,  0.03149748,  0.35885256)
	  rgb=(0.27495242,  0.03775181,  0.36454323)
	  rgb=(0.27602238,  0.04416723,  0.37016418)
	  rgb=(0.2770184 ,  0.05034437,  0.37571452)
	  rgb=(0.27794143,  0.05632444,  0.38119074)
	  rgb=(0.27879067,  0.06214536,  0.38659204)
	  rgb=(0.2795655 ,  0.06783587,  0.39191723)
	  rgb=(0.28026658,  0.07341724,  0.39716349)
	  rgb=(0.28089358,  0.07890703,  0.40232944)
	  rgb=(0.28144581,  0.0843197 ,  0.40741404)
	  rgb=(0.28192358,  0.08966622,  0.41241521)
	  rgb=(0.28232739,  0.09495545,  0.41733086)
	  rgb=(0.28265633,  0.10019576,  0.42216032)
	  rgb=(0.28291049,  0.10539345,  0.42690202)
	  rgb=(0.28309095,  0.11055307,  0.43155375)
	  rgb=(0.28319704,  0.11567966,  0.43611482)
	  rgb=(0.28322882,  0.12077701,  0.44058404)
	  rgb=(0.28318684,  0.12584799,  0.44496   )
	  rgb=(0.283072  ,  0.13089477,  0.44924127)
	  rgb=(0.28288389,  0.13592005,  0.45342734)
	  rgb=(0.28262297,  0.14092556,  0.45751726)
	  rgb=(0.28229037,  0.14591233,  0.46150995)
	  rgb=(0.28188676,  0.15088147,  0.46540474)
	  rgb=(0.28141228,  0.15583425,  0.46920128)
	  rgb=(0.28086773,  0.16077132,  0.47289909)
	  rgb=(0.28025468,  0.16569272,  0.47649762)
	  rgb=(0.27957399,  0.17059884,  0.47999675)
	  rgb=(0.27882618,  0.1754902 ,  0.48339654)
	  rgb=(0.27801236,  0.18036684,  0.48669702)
	  rgb=(0.27713437,  0.18522836,  0.48989831)
	  rgb=(0.27619376,  0.19007447,  0.49300074)
	  rgb=(0.27519116,  0.1949054 ,  0.49600488)
	  rgb=(0.27412802,  0.19972086,  0.49891131)
	  rgb=(0.27300596,  0.20452049,  0.50172076)
	  rgb=(0.27182812,  0.20930306,  0.50443413)
	  rgb=(0.27059473,  0.21406899,  0.50705243)
	  rgb=(0.26930756,  0.21881782,  0.50957678)
	  rgb=(0.26796846,  0.22354911,  0.5120084 )
	  rgb=(0.26657984,  0.2282621 ,  0.5143487 )
	  rgb=(0.2651445 ,  0.23295593,  0.5165993 )
	  rgb=(0.2636632 ,  0.23763078,  0.51876163)
	  rgb=(0.26213801,  0.24228619,  0.52083736)
	  rgb=(0.26057103,  0.2469217 ,  0.52282822)
	  rgb=(0.25896451,  0.25153685,  0.52473609)
	  rgb=(0.25732244,  0.2561304 ,  0.52656332)
	  rgb=(0.25564519,  0.26070284,  0.52831152)
	  rgb=(0.25393498,  0.26525384,  0.52998273)
	  rgb=(0.25219404,  0.26978306,  0.53157905)
	  rgb=(0.25042462,  0.27429024,  0.53310261)
	  rgb=(0.24862899,  0.27877509,  0.53455561)
	  rgb=(0.2468114 ,  0.28323662,  0.53594093)
	  rgb=(0.24497208,  0.28767547,  0.53726018)
	  rgb=(0.24311324,  0.29209154,  0.53851561)
	  rgb=(0.24123708,  0.29648471,  0.53970946)
	  rgb=(0.23934575,  0.30085494,  0.54084398)
	  rgb=(0.23744138,  0.30520222,  0.5419214 )
	  rgb=(0.23552606,  0.30952657,  0.54294396)
	  rgb=(0.23360277,  0.31382773,  0.54391424)
	  rgb=(0.2316735 ,  0.3181058 ,  0.54483444)
	  rgb=(0.22973926,  0.32236127,  0.54570633)
	  rgb=(0.22780192,  0.32659432,  0.546532  )
	  rgb=(0.2258633 ,  0.33080515,  0.54731353)
	  rgb=(0.22392515,  0.334994  ,  0.54805291)
	  rgb=(0.22198915,  0.33916114,  0.54875211)
	  rgb=(0.22005691,  0.34330688,  0.54941304)
	  rgb=(0.21812995,  0.34743154,  0.55003755)
	  rgb=(0.21620971,  0.35153548,  0.55062743)
	  rgb=(0.21429757,  0.35561907,  0.5511844 )
	  rgb=(0.21239477,  0.35968273,  0.55171011)
	  rgb=(0.2105031 ,  0.36372671,  0.55220646)
	  rgb=(0.20862342,  0.36775151,  0.55267486)
	  rgb=(0.20675628,  0.37175775,  0.55311653)
	  rgb=(0.20490257,  0.37574589,  0.55353282)
	  rgb=(0.20306309,  0.37971644,  0.55392505)
	  rgb=(0.20123854,  0.38366989,  0.55429441)
	  rgb=(0.1994295 ,  0.38760678,  0.55464205)
	  rgb=(0.1976365 ,  0.39152762,  0.55496905)
	  rgb=(0.19585993,  0.39543297,  0.55527637)
	  rgb=(0.19410009,  0.39932336,  0.55556494)
	  rgb=(0.19235719,  0.40319934,  0.55583559)
	  rgb=(0.19063135,  0.40706148,  0.55608907)
	  rgb=(0.18892259,  0.41091033,  0.55632606)
	  rgb=(0.18723083,  0.41474645,  0.55654717)
	  rgb=(0.18555593,  0.4185704 ,  0.55675292)
	  rgb=(0.18389763,  0.42238275,  0.55694377)
	  rgb=(0.18225561,  0.42618405,  0.5571201 )
	  rgb=(0.18062949,  0.42997486,  0.55728221)
	  rgb=(0.17901879,  0.43375572,  0.55743035)
	  rgb=(0.17742298,  0.4375272 ,  0.55756466)
	  rgb=(0.17584148,  0.44128981,  0.55768526)
	  rgb=(0.17427363,  0.4450441 ,  0.55779216)
	  rgb=(0.17271876,  0.4487906 ,  0.55788532)
	  rgb=(0.17117615,  0.4525298 ,  0.55796464)
	  rgb=(0.16964573,  0.45626209,  0.55803034)
	  rgb=(0.16812641,  0.45998802,  0.55808199)
	  rgb=(0.1666171 ,  0.46370813,  0.55811913)
	  rgb=(0.16511703,  0.4674229 ,  0.55814141)
	  rgb=(0.16362543,  0.47113278,  0.55814842)
	  rgb=(0.16214155,  0.47483821,  0.55813967)
	  rgb=(0.16066467,  0.47853961,  0.55811466)
	  rgb=(0.15919413,  0.4822374 ,  0.5580728 )
	  rgb=(0.15772933,  0.48593197,  0.55801347)
	  rgb=(0.15626973,  0.4896237 ,  0.557936  )
	  rgb=(0.15481488,  0.49331293,  0.55783967)
	  rgb=(0.15336445,  0.49700003,  0.55772371)
	  rgb=(0.1519182 ,  0.50068529,  0.55758733)
	  rgb=(0.15047605,  0.50436904,  0.55742968)
	  rgb=(0.14903918,  0.50805136,  0.5572505 )
	  rgb=(0.14760731,  0.51173263,  0.55704861)
	  rgb=(0.14618026,  0.51541316,  0.55682271)
	  rgb=(0.14475863,  0.51909319,  0.55657181)
	  rgb=(0.14334327,  0.52277292,  0.55629491)
	  rgb=(0.14193527,  0.52645254,  0.55599097)
	  rgb=(0.14053599,  0.53013219,  0.55565893)
	  rgb=(0.13914708,  0.53381201,  0.55529773)
	  rgb=(0.13777048,  0.53749213,  0.55490625)
	  rgb=(0.1364085 ,  0.54117264,  0.55448339)
	  rgb=(0.13506561,  0.54485335,  0.55402906)
	  rgb=(0.13374299,  0.54853458,  0.55354108)
	  rgb=(0.13244401,  0.55221637,  0.55301828)
	  rgb=(0.13117249,  0.55589872,  0.55245948)
	  rgb=(0.1299327 ,  0.55958162,  0.55186354)
	  rgb=(0.12872938,  0.56326503,  0.55122927)
	  rgb=(0.12756771,  0.56694891,  0.55055551)
	  rgb=(0.12645338,  0.57063316,  0.5498411 )
	  rgb=(0.12539383,  0.57431754,  0.54908564)
	  rgb=(0.12439474,  0.57800205,  0.5482874 )
	  rgb=(0.12346281,  0.58168661,  0.54744498)
	  rgb=(0.12260562,  0.58537105,  0.54655722)
	  rgb=(0.12183122,  0.58905521,  0.54562298)
	  rgb=(0.12114807,  0.59273889,  0.54464114)
	  rgb=(0.12056501,  0.59642187,  0.54361058)
	  rgb=(0.12009154,  0.60010387,  0.54253043)
	  rgb=(0.11973756,  0.60378459,  0.54139999)
	  rgb=(0.11951163,  0.60746388,  0.54021751)
	  rgb=(0.11942341,  0.61114146,  0.53898192)
	  rgb=(0.11948255,  0.61481702,  0.53769219)
	  rgb=(0.11969858,  0.61849025,  0.53634733)
	  rgb=(0.12008079,  0.62216081,  0.53494633)
	  rgb=(0.12063824,  0.62582833,  0.53348834)
	  rgb=(0.12137972,  0.62949242,  0.53197275)
	  rgb=(0.12231244,  0.63315277,  0.53039808)
	  rgb=(0.12344358,  0.63680899,  0.52876343)
	  rgb=(0.12477953,  0.64046069,  0.52706792)
	  rgb=(0.12632581,  0.64410744,  0.52531069)
	  rgb=(0.12808703,  0.64774881,  0.52349092)
	  rgb=(0.13006688,  0.65138436,  0.52160791)
	  rgb=(0.13226797,  0.65501363,  0.51966086)
	  rgb=(0.13469183,  0.65863619,  0.5176488 )
	  rgb=(0.13733921,  0.66225157,  0.51557101)
	  rgb=(0.14020991,  0.66585927,  0.5134268 )
	  rgb=(0.14330291,  0.66945881,  0.51121549)
	  rgb=(0.1466164 ,  0.67304968,  0.50893644)
	  rgb=(0.15014782,  0.67663139,  0.5065889 )
	  rgb=(0.15389405,  0.68020343,  0.50417217)
	  rgb=(0.15785146,  0.68376525,  0.50168574)
	  rgb=(0.16201598,  0.68731632,  0.49912906)
	  rgb=(0.1663832 ,  0.69085611,  0.49650163)
	  rgb=(0.1709484 ,  0.69438405,  0.49380294)
	  rgb=(0.17570671,  0.6978996 ,  0.49103252)
	  rgb=(0.18065314,  0.70140222,  0.48818938)
	  rgb=(0.18578266,  0.70489133,  0.48527326)
	  rgb=(0.19109018,  0.70836635,  0.48228395)
	  rgb=(0.19657063,  0.71182668,  0.47922108)
	  rgb=(0.20221902,  0.71527175,  0.47608431)
	  rgb=(0.20803045,  0.71870095,  0.4728733 )
	  rgb=(0.21400015,  0.72211371,  0.46958774)
	  rgb=(0.22012381,  0.72550945,  0.46622638)
	  rgb=(0.2263969 ,  0.72888753,  0.46278934)
	  rgb=(0.23281498,  0.73224735,  0.45927675)
	  rgb=(0.2393739 ,  0.73558828,  0.45568838)
	  rgb=(0.24606968,  0.73890972,  0.45202405)
	  rgb=(0.25289851,  0.74221104,  0.44828355)
	  rgb=(0.25985676,  0.74549162,  0.44446673)
	  rgb=(0.26694127,  0.74875084,  0.44057284)
	  rgb=(0.27414922,  0.75198807,  0.4366009 )
	  rgb=(0.28147681,  0.75520266,  0.43255207)
	  rgb=(0.28892102,  0.75839399,  0.42842626)
	  rgb=(0.29647899,  0.76156142,  0.42422341)
	  rgb=(0.30414796,  0.76470433,  0.41994346)
	  rgb=(0.31192534,  0.76782207,  0.41558638)
	  rgb=(0.3198086 ,  0.77091403,  0.41115215)
	  rgb=(0.3277958 ,  0.77397953,  0.40664011)
	  rgb=(0.33588539,  0.7770179 ,  0.40204917)
	  rgb=(0.34407411,  0.78002855,  0.39738103)
	  rgb=(0.35235985,  0.78301086,  0.39263579)
	  rgb=(0.36074053,  0.78596419,  0.38781353)
	  rgb=(0.3692142 ,  0.78888793,  0.38291438)
	  rgb=(0.37777892,  0.79178146,  0.3779385 )
	  rgb=(0.38643282,  0.79464415,  0.37288606)
	  rgb=(0.39517408,  0.79747541,  0.36775726)
	  rgb=(0.40400101,  0.80027461,  0.36255223)
	  rgb=(0.4129135 ,  0.80304099,  0.35726893)
	  rgb=(0.42190813,  0.80577412,  0.35191009)
	  rgb=(0.43098317,  0.80847343,  0.34647607)
	  rgb=(0.44013691,  0.81113836,  0.3409673 )
	  rgb=(0.44936763,  0.81376835,  0.33538426)
	  rgb=(0.45867362,  0.81636288,  0.32972749)
	  rgb=(0.46805314,  0.81892143,  0.32399761)
	  rgb=(0.47750446,  0.82144351,  0.31819529)
	  rgb=(0.4870258 ,  0.82392862,  0.31232133)
	  rgb=(0.49661536,  0.82637633,  0.30637661)
	  rgb=(0.5062713 ,  0.82878621,  0.30036211)
	  rgb=(0.51599182,  0.83115784,  0.29427888)
	  rgb=(0.52577622,  0.83349064,  0.2881265 )
	  rgb=(0.5356211 ,  0.83578452,  0.28190832)
	  rgb=(0.5455244 ,  0.83803918,  0.27562602)
	  rgb=(0.55548397,  0.84025437,  0.26928147)
	  rgb=(0.5654976 ,  0.8424299 ,  0.26287683)
	  rgb=(0.57556297,  0.84456561,  0.25641457)
	  rgb=(0.58567772,  0.84666139,  0.24989748)
	  rgb=(0.59583934,  0.84871722,  0.24332878)
	  rgb=(0.60604528,  0.8507331 ,  0.23671214)
	  rgb=(0.61629283,  0.85270912,  0.23005179)
	  rgb=(0.62657923,  0.85464543,  0.22335258)
	  rgb=(0.63690157,  0.85654226,  0.21662012)
	  rgb=(0.64725685,  0.85839991,  0.20986086)
	  rgb=(0.65764197,  0.86021878,  0.20308229)
	  rgb=(0.66805369,  0.86199932,  0.19629307)
	  rgb=(0.67848868,  0.86374211,  0.18950326)
	  rgb=(0.68894351,  0.86544779,  0.18272455)
	  rgb=(0.69941463,  0.86711711,  0.17597055)
	  rgb=(0.70989842,  0.86875092,  0.16925712)
	  rgb=(0.72039115,  0.87035015,  0.16260273)
	  rgb=(0.73088902,  0.87191584,  0.15602894)
	  rgb=(0.74138803,  0.87344918,  0.14956101)
	  rgb=(0.75188414,  0.87495143,  0.14322828)
	  rgb=(0.76237342,  0.87642392,  0.13706449)
	  rgb=(0.77285183,  0.87786808,  0.13110864)
	  rgb=(0.78331535,  0.87928545,  0.12540538)
	  rgb=(0.79375994,  0.88067763,  0.12000532)
	  rgb=(0.80418159,  0.88204632,  0.11496505)
	  rgb=(0.81457634,  0.88339329,  0.11034678)
	  rgb=(0.82494028,  0.88472036,  0.10621724)
	  rgb=(0.83526959,  0.88602943,  0.1026459 )
	  rgb=(0.84556056,  0.88732243,  0.09970219)
	  rgb=(0.8558096 ,  0.88860134,  0.09745186)
	  rgb=(0.86601325,  0.88986815,  0.09595277)
	  rgb=(0.87616824,  0.89112487,  0.09525046)
	  rgb=(0.88627146,  0.89237353,  0.09537439)
	  rgb=(0.89632002,  0.89361614,  0.09633538)
	  rgb=(0.90631121,  0.89485467,  0.09812496)
	  rgb=(0.91624212,  0.89609127,  0.1007168 )
	  rgb=(0.92610579,  0.89732977,  0.10407067)
	  rgb=(0.93590444,  0.8985704 ,  0.10813094)
	  rgb=(0.94563626,  0.899815  ,  0.11283773)
	  rgb=(0.95529972,  0.90106534,  0.11812832)
	  rgb=(0.96489353,  0.90232311,  0.12394051)
	  rgb=(0.97441665,  0.90358991,  0.13021494)
	  rgb=(0.98386829,  0.90486726,  0.13689671)
	  rgb=(0.99324789,  0.90615657,  0.1439362 )
	}
  }
}

\usetikzlibrary{calc}
\usetikzlibrary{intersections}
\usepgfplotslibrary{fillbetween}
\usetikzlibrary{external}

\makeatletter
\renewcommand{\todo}[2][]{\tikzexternaldisable\@todo[#1]{#2}\tikzexternalenable}
\renewcommand{\missingfigure}[2][]{\tikzexternaldisable\@missingfigure[#1]{#2}\tikzexternalenable}
\makeatother

\makeatletter
\newtheoremstyle{indented}
  {3pt}
  {3pt}
  {\addtolength{\@totalleftmargin}{2em}
   \addtolength{\linewidth}{-2em}
   \parshape 1 2em \linewidth}
  {}
  {\bfseries}
  {.}
  {.5em}
  {}
\makeatother
\theoremstyle{indented}
\newtheorem{theorem}{Theorem}[section]

\newtheorem{lemma}[theorem]{Lemma}
\newtheorem{proposition}[theorem]{Proposition}
\newtheorem{definition}[theorem]{Definition}
\newtheorem{remark}[theorem]{Remark}
\newtheorem{assumption}[theorem]{Assumption}

\newcommand{\argmax}{\operatornamewithlimits{arg\,max}}

\newcommand{\eqdef}{:=}
\newcommand{\eqfed}{=:}
\newcommand{\R}{\mathbb{R}}
\newcommand{\Rnneg}{\mathbb{R}_{\geq 0}}
\newcommand{\N}{\mathbb{N}}
\newcommand{\E}{\mathbb{E}}
\newcommand{\eps}{\epsilon}
\newcommand{\interior}[1]{%
  {\kern0pt#1}^{\mathrm{o}}%
}

\newcommand{\process}[1]{(#1_t)_{t \geq 0}}

\newcommand{\ccontr}{\alpha}
\newcommand{\dcontr}{\beta}
\newcommand{\ccost}{F}
\newcommand{\dcost}{G}
\newcommand{\discRate}{\rho}
\newcommand{\cdProc}{X^{\ccontr, \dcontr}}
\newcommand{\cProc}{X^{\ccontr, 0}}

\newcommand{\dOp}{\mathcal{D}}
\newcommand{\cOp}{\mathcal{L}}
\newcommand{\iterOp}{\mathcal{T}}

\newcommand{\functionSpace}{\mathcal{X}}
\newcommand{\metric}{d}
\newcommand{\ruinTime}{\theta}
\newcommand{\pcost}{\lambda_p}
\newcommand{\fcost}{\lambda_f}
\newcommand{\filtration}{\mathcal{F}}

\newcommand{\xsigma}{\sigma}
\newcommand{\xW}{W}
\newcommand{\mumean}{\bar{\mu}}

\newcommand{\musigma}{\tilde{\sigma}}
\newcommand{\muW}{\tilde{W}}
\newcommand{\alphaMetric}{d_\alpha}
\newcommand{\cashflowGenerator}{\mathcal{A}}

\newcommand*{\wcc}{\skew{3}{\widehat}{\mathcal{C}}}

\def\cA{{\mathcal A}}
\def\eps{\epsilon}

\def\Om{\Omega}
\def\P{\mathbb{P}}
\makeatletter
\def\1{\@ifnextchar[{\@indicatorset}{\@indicator}}
\def\@indicatorset[#1]{\mathbf{1}_{\{#1\}}}
\def\@indicator{\mathbf{1}}
\makeatother

\def\cA{{\mathcal A}}

\def\cC{{\mathcal C}}

\def\cI{{\mathcal I}}

\def\R{{\mathcal R}}

\def\E{\mathbb{E}}

\def\N{\mathbb{N}}

\def\P{\mathbb{P}}
\def\R{\mathbb{R}}

\newcommand{\ulf}{\phi_*}
\newcommand{\olf}{\phi^*}
\newcommand{\vzero}{v_{0}}
\newcommand{\vbzero}{\bar{v}_{0}}

\makeatletter
\def\vn{\@ifnextchar[{\@vnopt}{\@vn}}
\def\@vnopt[#1]{v_{#1}}
\def\@vn{v_{n}}
\makeatother

\newcommand{\vbn}{\bar{v}_n}
\newcommand{\ctwo}{C^{1,2}(Q)}
\newcommand{\cinfty}{C^{\infty}(Q)}
\newcommand{\ccont}{C(\overline{Q})}

\numberwithin{equation}{section}
\title{Discrete dividend payments in continuous time}

\author{
	Jussi Keppo\thanks{NUS Business School and Institute of Operations Research and Analytics, National University of Singapore, email: \href{mailto:keppo@nus.edu.sg}{\texttt{keppo@nus.edu.sg}}. Partly supported by Institute of Operations Research and Analytics (National University of Singapore) grant WBS-R-726-000-009-646.}
	\and
    A.\ Max Reppen\thanks{ORFE Department, Princeton University, Princeton, NJ 08544, USA, email: \href{mailto:areppen@princeton.edu}{\texttt{areppen@princeton.edu}}. Supported by the Swiss National Science Foundation grant SNF 181815.}{ }\footnotemark[4]
    \and
    H.\ Mete Soner\thanks{ETH Z\"urich, Department of Mathematics, R\"amistrasse 101, 8092 Z\"urich, Switzerland, and Swiss Finance Institute, email: \href{mailto:mete.soner@math.ethz.ch}{\texttt{mete.soner@math.ethz.ch}}.}{ }\thanks{Partly supported by the ETH Foundation, the Swiss Finance Institute, and Swiss National Foundation grant SNF 200020\_172815.}
}

\date{\today}

\begin{document}

\maketitle

\begin{abstract}
\noindent
We propose a model in which dividend payments occur at regular, deterministic intervals in an otherwise continuous model.
This contrasts traditional models where either the payment of continuous dividends is controlled or the dynamics are given by discrete time processes.
Moreover, between two dividend payments, the structure allows for other types of control; we consider the possibility of equity issuance at any point in time.
The value is characterized as the fixed point of an optimal control problem with periodic initial and terminal conditions.
We prove the regularity and uniqueness of the corresponding dynamic programming equation,
and the convergence of an efficient numerical algorithm that we use to study the problem.
The model enables us to find the loss caused by infrequent dividend payments.
We show that under realistic parameter values this loss varies from around 1\% to 24\% depending on the state of the system, and that using the optimal policy from the continuous problem further increases the loss.
\end{abstract}

\section{Introduction}
Continuous time decision making is prevalent and of great importance, but some phenomena only occur at discrete intervals.
In models for asset trading, these intervals are typically sufficiently short to justify a continuous time model.
However, other types of events take place on larger time scales and thus weakens the basis for a continuous time approximation.
One such example is dividend payments that we study in this paper.
Variation of the dividend frequency typically ranges from monthly to annually, far less frequently than, for instance, trading on a large exchange.

Our approach for tackling this discrepancy between the models and practise is to consider a continuous time model in which dividends may only be paid out at predetermined, discrete time points.
This is reminiscent of models where the dividend payment times are discrete and random, as seen in \cite{avanzi2016interface}, but the mathematical structure is significantly different in the random and predetermined models.
Moreover, the distinction between our model and a traditional discrete time model lies in the possibility to model other continuous decision-making problems in-between the dividend payments.
In particular, we allow for the possibility to issue equity.
Since we do not wish to restrict equity issuance to predetermined time points---as would be the case in a discrete time model---we model this issuance as a continuous time control problem.
In other words, equity can be issued at any point in time.
Although our choice of continuous time control is equity issuance, the same methodology can also be used for other types of decisions and models, for instance capital investments.
Notwithstanding, the model at hand is not only a good example for the method, but our numerical results facilitates an interesting comparison to the continuous time counterpart.
More specifically, we show that under realistic parameter values the difference between the value function of discrete time dividends and the corresponding value function with continuous time dividends is often around 1--3\% relative to the continuous dividend case, but increases to 24\% depending on the state of the system.

Moreover, this type of structure does not only appear in problems with control decisions at discrete times.
In fact, problems in which monitoring occurs at discrete time points readily fit into the same framework.
One example of such a model is that of leveraged exchange traded funds (leveraged ETFs or LETFs).
The goal of an LETF is to track the returns of some index on some time scale---typically daily---by a predetermined multiple.
In \cite{LETFs}, the authors model this tracking problem by imposing a `monitoring condition' at the end of each trading day;
at the end of the day, it incurs a penalty depending how closely it tracks the underlying index.
From a structural point of view, the trading and transaction cost payment, happening in continuous time, is akin to the equity issuance, whereas the `monitoring' takes the role of the dividend payment (despite not being actively controlled).

Our main focus is to characterize the value function as the solution to a parabolic PDE with a fixed point structure.
To numerically find a solution, one needs to iteratively solve a related control problem without the discrete time element, i.e., without the dividends/monitoring.
In our model, we do not otherwise make any specific assumptions on the cash flow process other than that the cash flow process cannot be too large.
In particular, the results hold for both the commonly studied jump models, cf.~\cite{gerber1969entscheidungskriterien}, as well as their diffusion model counterparts, cf.\ e.g.\ \cite{jeanblanc1995,gerber2004optimal}.

In the context of dividend problems specifically, one common point of criticism of many optimal dividend problems is their irregular dividend payments when following the optimal policy.
One way to alleviate this is to consider dividend policies that are proportional or affine as a function of the current reserves, cf.~\cite{avanzi2012mean,albrecher2017risk}.
Although we do not explicitly consider any such models, they still fit naturally into the framework presented in this paper.
For further references on optimal dividend problems, we refer the reader to \cite{asmussen2010ruin, albrecher2009optimality}.

The structure of the paper is given as follows:
For the purpose of showcasing the main idea, we begin with a nonrigorous description of the general structure in Section~\ref{sec:generalStructure}.
We thereafter give an account of how an optimal dividend problem with continuous issuance and discrete dividends fit into this framework in Section~\ref{sec:dividends1D}.  We also prove a regularity
result for the one period problem with only equity issuance as control, which could be of interest.
In this context we present the structure of the main equation and describe a numerical scheme for finding its solution.
In Section~\ref{sec:dividends2D}, we extend the optimal dividend model and the convergence results to a multidimensional model in the spirit of \cite{reppen2017dividends}, but here with discretized dividend payments.
Numerical studies on the value and policy impact of dividend discretization is conducted in both settings.
Finally, Section~\ref{sec:conclusion} provides a summary of our findings along with our interpretations and conclusions.

\section{General structure}
\label{sec:generalStructure}

Although the focus of the paper is on optimal dividend problems, the core idea extends to a
wider class of problems and is best showcased in a general setting.
What follows in this section is a formal discussion of the ideas that will later be made
rigorous for two dividend problems.

We consider a specific type of infinite horizon (possibly singular) stochastic optimal
control problem with discounting.
What distinguishes these problems is that they, at regular, equidistant intervals,
involve a singular action and/or monitoring or a singular jump in the payoff function.
In this sense, the structure can be considered as a mix of continuous and discrete time control problems.

The structure of the problem can thus be separated into two components:
one for what happens at the discrete time points and one for what happens in-between.
For simplicity, we will represent these components by incremental operators that represent the
equations determining the solution across these time regions.

In particular, we allow for two controls $\ccontr$ and $\dcontr$
 that represent the continuous control and the discrete control respectively.
For a given choice of $\ccontr$ and $\dcontr$, we denote
by $\cdProc = \process{\cdProc}$ the state process corresponding to this choice.
It is here implicit that the process does not depend on $\dcontr$ in-between the discrete time points.
Similarly, we consider two types of cost structures: $\ccost^\ccontr_t = \ccost(\ccontr_t, \cdProc_t)$
for the cumulative (undiscounted) continuous cost and $\dcost^\dcontr_t = \dcost(\dcontr_t, \cdProc_{t-})$
for the cost incurred at the discrete points. Note that $\ccontr$ may be a singular control process.
Finally, let $T$ be the time between two of the discrete time points.

With this structure, we can write the control problem as
\[
	V(x) = \sup_{\ccontr, \dcontr}
	\E_x \bigg[ \int_0^\infty e^{-\discRate t} \dif{\ccost}^\ccontr_t +
	\sum_{n=0}^\infty e^{-\discRate n} \dcost^\dcontr_n\bigg],
\]
where $\E_x$ denotes expectation with respect to a measure under which
the state process starts at $x$ (before any control is activated), the
supremum is over some set of admissible controls, and $\discRate$ is the
discounting rate.
To proceed, we require the \emph{discrete time dynamic programming principle}
(DTDPP) to hold, i.e., dynamic programming at the discrete time points $t \in T\N = \{0, T, 2T, \dots \}$.
This and that the value function is universally measurable are established by
\citet{bertsekas_stochastic_1978}.

Now suppose there exists a space $\functionSpace$ such that the composition
of the following two `incremental' operators are well-defined.
First, the continuous operator is given by
\[
	\cOp \phi (x) = \sup_{\ccontr} \E_x \bigg[ \int_0^T e^{-\discRate t}
	\dif\ccost^\ccontr_t + e^{-\discRate T} \phi(\cProc_{T-}) \bigg].
\]
Second, the discrete operator is given by
\[
	\dOp \phi(x) = \sup_{\dcontr} \big( \phi(x + \dcontr) + \dcost(\dcontr, x) \big).
\]
In other words, there exists a space $\functionSpace$ such that $\dOp \circ \cOp :
\functionSpace \to \functionSpace$ or
$\cOp \circ \dOp : \functionSpace \to \functionSpace$.\footnote{Note that the
first operator in these compositions may map to an intermediate space.}
Without loss of generality, assume that it holds for $\iterOp \eqdef \dOp \circ \cOp$.
Note that the universal measurability of the value function $V$ makes the operators
well-defined on $V$, and the DTDPP states precisely that $V = \iterOp V$.
Our goal is to show that the value function can be found by iteratively applying $\iterOp$.
To do so, we seek a complete metric space $(\functionSpace, d)$ such that $\iterOp$ is
a strict contraction and $V \in \functionSpace$.
If such a space exists, $\iterOp$ will have a (unique) fixed point, provided the space is not empty.
Then, by the DTDPP, the value function is the fixed point, and
$\lim_{n \to \infty} (\dOp \circ \cOp)^n \phi = V$ for every
$\phi \in \functionSpace$.\footnote{In fact, this is in fact so-called
value iteration \cite{Bel}.}

In the rest of this paper, we will show how $\functionSpace$ and $\metric$ can be
chosen for the two optimal dividend problems.

\section{Discrete dividend payments with capital injections}
\label{sec:dividends1D}
This section is devoted to the optimal dividend problem for which the fixed point idea
from Section~\ref{sec:generalStructure} can be applied.
An equity capital constrained firm pays dividends to its shareholders at discrete,
predetermined time intervals.
The firm may also choose to issue equity at any point in time.

\subsection{Problem}
\label{ss.problem}

Consider a  firm endowed with some cash flow
that are placed in the firm's \emph{cash reserves}.
Dividends may be paid from these reserves periodically
at times $T,2T,\ldots$
until the time of ruin or bankruptcy.
Simultaneously, it can issue equity at any time
of its choice to increase its reserves and avoid ruin.
The aim of the firm is to maximize the discounted
value of dividends, net of capital injections.

To formulate the problem mathematically,
let $(\Om, \P)$ be a probability space
with a filtration $\filtration = \process{\filtration}$
satisfying the usual conditions \cite{protter}
and $W$ be a one-dimensional Brownian motion.
Let $x$ denote the initial cash reserve,
$L = \process{L}$ the cumulative dividends paid, and
$I = \process{I}$ the cumulative equity issued.
We refer to these two processes as dividend and issuance policies, strategies, and controls interchangeably.
We assume that  $(L,I)$ are RCLL (right-continuous with left-limits),
and are adapted to $\filtration$.
Then, the net cash reserves $X^\nu := \process{X^\nu}$
are given by,
\[
X^\nu_t :=x+C_t - L_t + I_t, \quad
{\text{with}} \quad
C_t= \mu t + \sigma W_t,
\]
where $\nu=(x,L,I)$ denotes the dependence
on these processes, given positive constants $\sigma$ and $\mu$, $\sigma$ is the volatility and $\mu$ is the profitability of the firm.
Note that $X^\nu$ is also RCLL and the initial
condition is interpreted as $X^\nu_{0-} = x$.
Similarly, we also think that $L_{0-}=I_{0-}=0$.

Dividends are paid from reserves periodically at times
$T,2T,\ldots$ until the
time of ruin,
\[
\ruinTime^\nu := \inf \{ t > 0 : X^\nu_t < 0 \}.
\]
The firm
closes its operations after ruin, and hence
all admissible dividend and issuance policies should satisfy
$\Delta L_t=\Delta I_t=0$ for all $t >\ruinTime^\nu$, where
for an RCLL process $Y$, $\Delta Y_t:= Y_t -Y_{t-}$.
Also, dividends must be fully covered by reserves of the firm.
This imposes the condition $\Delta L_t \leq X^\nu_{t-}$,
for all $t \geq 0$. In particular, as $\process{C}$ is continuous,
$X^\nu_{\ruinTime^\nu}=0$.  Additionally,
each non-zero issuance results in a fixed
cost, and thus there can  be at most finitely many
such actions in a finite time.  Hence,
for an admissible issuance process $I$,
the set $\{t \geq0 : \Delta I_t>0\}$ is countable.
We call  any pair of processes $(L,I)$ satisfying
these requirements  {\em{admissible}} and $ \cA_x$ is the set
of all admissible processes.

The aim of the firm is to maximize the discounted value of dividends, net of capital injections.
We follow \cite{decamps2011free} to model the cost of
equity issuance and with a given
discounting rate of $\discRate > 0$, fixed and
proportional issuance costs $\fcost > 0$ and $\pcost > 0$ respectively,
the value of the firm is given by
\[
V(x) := \sup_{(L,I)\in  \cA_x} J(\nu),\quad
	J(\nu):= \E_x \Bigg[
	\sum_{n=0}^\infty
	e^{-\discRate nT} \Delta L_t
	- \sum_{t \geq 0} e^{-\discRate t} c(\Delta I_t) \Bigg],
\]
where as before $\nu=(x,L,I)$ and   $c(\zeta):=	(\fcost + (1 + \pcost) \zeta) \1_{\{\zeta> 0\}}$,
$\E_x$ denotes expectation conditioned on $X^{\nu}_{0-} = x$.
For continuously paid dividends, this problem was first
formulated in \cite{decamps2011free}
and later further
studied in \cite{akyildirim2014optimal, reppen2017dividends}.

\begin{remark}
The strict positivity of $\pcost$ is used
{\em{only}} in the proof
Lemma \ref{p.time} below, to show that
any issuance is necessarily bounded.  Therefore,
our proof would also be valid for
models with $\pcost=0$ but the
issuance sizes are bounded
by a given constant.  This is the case
in our numerical examples.
In fact, more tedious analysis
yields the result
without any  assumption
on $\pcost$. However,
to simplify the presentation
we choose to assume $\pcost>0$.
 \end{remark}

\subsection{Fixed-point Structure}
\label{sec:periodization_1D}
For $\phi : \Rnneg \to \mathbb{R}_{\geq 0}$ and $x\ge 0$, set
$\cI(\phi)(x):= \sup_{\zeta > 0} \left(\phi(x+\zeta)-c(\zeta)\right)$.

\begin{definition}
\label{d.ca}
Define the following spaces.
\begin{itemize}[noitemsep, topsep=0pt]
    \item
        Let $\wcc_A$ be the set of  all continuous, non-decreasing
        functions $\phi : \Rnneg \to \mathbb{R}_{\geq 0}$ satisfying
        $(\cI(\phi)(x))^+ \le \phi(x) \le x+A $ and
        \begin{equation}
            \label{e.1}
            \phi(0) = \cI(\phi)(0) \vee 0 =: (\cI(\phi)(0))^+.
        \end{equation}
    \item Let $\cC_A$ be the set of all $\phi \in \wcc_A$
        such that $\varphi(x):=\phi(x)-x$ is bounded by $A$ and non-decreasing on $\Rnneg$.
        In particular, this implies that every $\phi \in \cC_A$  is
        uniformly continuous on $\Rnneg$ and
        $x \le \phi(x) \le x+A$.
\end{itemize}
\end{definition}

Following the procedure outlined in Section~\ref{sec:generalStructure},
define $\cOp$, $\dOp$, $\iterOp$ on $\wcc_A$ by
\begin{align}
	\label{eqn:cOp1D}
	\cOp\phi(x) &\eqdef  \sup_{I \in \widehat \cA_x}
	 \E_x\big( - \sum_{\mathclap{0 \le s <T}} e^{-\discRate s} c(\Delta I_s)
	+ e^{-\discRate T}  \phi(X^{\alpha}_{T}) \1_{\{T < \ruinTime^{\alpha}\}}\big),\\
	\label{eqn:dOp1D}
	\dOp\phi(x) &\eqdef  \sup_{0\le \ell \leq x} ( \phi(x - \ell) + \ell ),
\end{align}
where for $\alpha=(x,I)$, $X^\alpha:=X^{(x,0,I)}$,
$\ruinTime^\alpha= \ruinTime^{(x,0,I)}$, and $I\in \widehat \cA_x$ provided
that $(0,I)\in \cA_x$.  Set $\iterOp \eqdef \dOp \circ \cOp$.

Note that all above operators are monotone.
Also, for any constant $c$ and $\phi \in \wcc_A$,
$\dOp(c+\phi)(x) =c+\dOp \phi (x)$ and
$\cOp(c+\phi)(x) \le e^{-\discRate T} c +\cOp \phi(x)$.

\begin{theorem}
	\label{thm:convergence_1D} For each $A\ge 0$,
	$\dOp :\wcc_A \to \cC_A$.
	Furthermore, there exists $A^*>0$ such that for every $A \ge A^*$,
	$\cOp :\cC_A \to \wcc_A$ and
	$\iterOp : \cC_A \to \cC_A$ is a strict contraction.
\end{theorem}
\begin{proof}
	Fix $A \geq 0$.
	Set $\ulf(x) \eqdef x$ for $x \ge 0$, and $\olf \eqdef A +\ulf$.

	{\em{Step 1.}}
	Fix $\phi \in \wcc_A$ and $\zeta \ge 0$. Since
	$c(\zeta) \ge \zeta$,
	for any $\ell \in [\zeta, x+\zeta]$,
    \[
        \phi(x+\zeta-\ell)+\ell \le \phi(x -(\ell-\zeta)) +(\ell-\zeta) +c(\zeta)
        \le \dOp \phi(x) + c(\zeta).
    \]
	Since $\phi \in \wcc_A$,
	$\phi(x+\hat \zeta) \le \phi(x)+ c(\hat \zeta)$ for any $x, \hat \zeta \ge 0$.
	Then for an arbitrary $\ell \in [0,\zeta]$, set $\hat \zeta= \zeta -\ell$ and use this inequality to obtain
    \[
        \phi(x+\zeta -\ell)+\ell \le \phi(x)+c(\zeta-\ell)+ \ell
        \le \phi(x)+c(\zeta) \le \dOp \phi(x) + c(\zeta),\quad \forall \ell \in [0,\zeta+ x].
    \]
	The above two inequalities imply that $\dOp \phi(x+\zeta)
	\le \dOp \phi(x) + c(\zeta)$ for every $x, \zeta \ge 0$. Since $\dOp \phi \ge0$,
	we conclude that
	$\dOp \phi \ge \cI(\dOp \phi)^+$.
	As $\phi \in \wcc_A$, $\phi(0)= \cI(\dOp \phi)(0)^+$.
	Additionally,  $\dOp\phi(0) =\phi(0)$ for any $\phi$.
	Hence, $\dOp \phi(0)=\cI(\dOp \phi)(0)^+$.

	For $h\ge 0$,
	\begin{align*}
	\dOp \phi(x+h) & =
	\sup_{0 \le \widehat \ell \le x+h} (\phi(x+h-\widehat \ell) +\widehat \ell)
	= \sup_{-h \le \ell=\widehat \ell -h \le x} (\phi(x-\ell) +\ell) +h\\
	& \ge \sup_{0\le \ell \le x} (\phi(x-\ell) +\ell)+h
	= \dOp \phi(x)+h.
	\end{align*}
	Hence, $\dOp\phi(x)-x$ is monotone.  Also,
	it is clear that $\dOp \phi$ is continuous
	and $\dOp \phi \le \dOp \olf  = \olf$.
	So we have proved that $\dOp $ maps $\wcc_A$ into $\cC_A$.

	{\em{Step 2.}} Recall that for $x \ge 0$ and $I \in \widehat \cA_x$,
	$\ruinTime^\alpha$ is the exit time and $\alpha=(x,I)$.  Set
	$\tau^\alpha:= \ruinTime^\alpha \wedge T$.
	Then, when $\ruinTime^\alpha \le T$, $X^\alpha_{\tau^\alpha}
	=X^\alpha_{\ruinTime^\alpha}=0$.
	Hence, $X^\alpha_T \1_{\{T < \ruinTime^\alpha\}} = X^\alpha_{\tau^\alpha}$.
	Since $\zeta + \tilde \zeta \le c(\zeta+\tilde \zeta) \le c(\zeta)+c(\tilde \zeta)$
	for every
	$\zeta, \tilde \zeta \ge 0$, and since $\Delta I_t=0$ for every $t > \ruinTime^\alpha$,
	\begin{align*}
	- \sum_{\mathclap{0 \le s <T}} e^{-\discRate s} c(\Delta I_t)
	+ e^{-\discRate T} X^\alpha_{T} \1_{\{T < \ruinTime^\alpha\}}
	&\le
	- \sum_{\mathclap{0 \le s <T}} e^{-\discRate T} c(\Delta I_t)
	+ e^{-\discRate T} X^\alpha_{\tau^\alpha}\\
	&\le
	- \sum_{\mathclap{0 \le s <T}} e^{-\discRate T} \Delta I_t
	+ e^{-\discRate T} X^\alpha_{\tau^\alpha}\\
	&= e^{-\discRate T}\left[-  I_{\tau^\alpha} + X^\alpha_{\tau^\alpha}\right]
	= e^{-\discRate T}\left[x + C_{\tau^\alpha}\right].
	\end{align*}
	This implies that
	$\cOp \ulf (x)
		 \le \E_x e^{-\discRate T}\left[x + C_{\tau^\alpha} \right]
		 \le e^{-\discRate T}(x+\mu T)$.
	Therefore,
	\[
	\cOp \olf(x)= \cOp (A+\ulf)(x) \le e^{-\discRate T}A +
	\cOp(\ulf)(x)  \le e^{-\discRate T}(A + x+\mu T)
	\le  e^{-\discRate T}x +A,
	\]
	provided that $e^{-\discRate T}(A +\mu T) \le A$ which is equivalent to
	$A \ge A^*:= \mu T (e^{\discRate T}-1) ^{-1}$.
	Hence, $0\le \cOp \phi(x)\le  e^{-\discRate T}x+A$ for every
	$x \ge 0$ and $\phi \in \cC_A$ whenever $A \ge A^*$.

	For $x\ge 0$, by making an immediate
	issuance of size $\zeta > 0$, we conclude that
	$\cOp \phi (x) \ge \cOp \phi(x+\zeta)-c(\zeta)$.
	Hence, $\cOp \phi \ge \cI(\cOp\phi)$.
	Theorem \ref{t.cont} below proves that $\cOp \phi$ is
	continuous for every $\phi \in \cC_A$ and satisfies \eqref{e.1}.
	Hence,  $\cOp$ maps $\cC_A$ into $\wcc_A$
	for all $A \ge A^*$.

	{\em{Step 3.}}  Fix   $A \ge A^*$.  Then, $\iterOp= \dOp \circ \cOp
	: \cC_A \to \cC_A$.
	We continue by showing that $\iterOp$ is a strict contraction with
	the metric $d(\phi, \varphi) := \sup_{x \ge 0} |\phi(x) - \varphi(x)|$.
	Indeed, for any $\phi, \varphi \in \cC_A$, $d(\phi,\varphi) \le A$ and
	$\cOp(\phi,\varphi)(x) \le e^{-\rho T} d(\phi,\varphi)$.  Hence,
	\[
	\sup_{x \ge 0}(\iterOp \phi - \iterOp \varphi)(x)
	\leq \sup_{x \ge 0} \sup_{0\le \ell \leq x} \cOp (\phi-\varphi)(x-\ell)
	\le e^{-\rho T} d(\phi-\varphi).
	\]
	Therefore,  $d(\iterOp f, \iterOp g) \leq e^{-\discRate T} d(f, g)$
	and $\iterOp$ is a strict contraction.
\end{proof}

The following verification result is the
main characterization of the value function.
It also provides a computational method
provided that an efficient method
for the operator $\cOp$ is constructed.
However, we emphasize that
Theorem \ref{thm:convergence_1D} above
and consequently the below verification
theorem use crucially
the regularity of the operator
$\cOp$ that will be proved in the
next subsection in Theorem~\ref{t.cont}.

\begin{theorem}[Verification]
\label{t.verification}
The value function $V \in \cC_{A^*}$ is the unique fixed point of $\iterOp$.
\end{theorem}

\begin{proof}

Since $\iterOp$ is a strict contraction on $\cC_{A^*}$, it has a unique
fixed-point $\Phi \in \cC_{A^*}$.  The definitions of $\cOp$ and $\dOp$ imply that
\[
\Phi(x)= \sup_{(L,I)\in \cA_x}
\E_x \Big[ \Delta L_0 - \sum_{0 \le t \leq T} e^{-\discRate t} c(\Delta I_t)
+ e^{-\discRate T}\Phi(X^\alpha_T)\1_{\{ \theta^\alpha>T\}}\Big].
\]
Since $\Phi$ is continuous, standard selection theorems imply
that
\begin{align*}
\Phi(x)&= \sup_{(L,I)\in \cA_x}
\E_x \Big[ \Delta L_0 +e^{-\discRate T}\Delta L_T
- \sum_{\mathclap{0 \le t \leq 2T}} e^{-\discRate t} c(\Delta I_t)
+ e^{-\discRate T}\Phi(X^\alpha_{2T})\1_{\{ \theta^\alpha>2T\}}\Big]\\
&=\sup_{(L,I)\in \cA_x}\E_x \Big[ \sum_{n=0}^{N-1}
	e^{-\discRate nT} \Delta L_{nT}
	- \sum_{\mathclap{0 \le t \leq (N-1)T}} e^{-\discRate t} c(\Delta I_t)
	+ e^{-\discRate NT}\Phi(X^\alpha_{NT})
\1_{\{ \theta^\alpha> NT\}} \Big].
\end{align*}
We now use the upper bound $\Phi(x) \le a+A^*$ to pass
to the limit as $N$ tends to infinity to show that
$\Phi=V$.
\end{proof}

\subsection{One-period problem}
\label{ss.oneperiod}

In this subsection,
we prove the properties of the
operator $\cOp$ used in Theorem \ref{thm:convergence_1D}
and provide  a technique for its computation.
We fix  $A\ge A^*$, $\phi \in \cC_A$
and set $Q:= (0,T) \times (0,\infty)$.

The nonlinear operator $\cOp \phi$  corresponds
to an optimal issuance problem
with a given terminal data.  As we study this problem
by dynamic programming,
we need to consider it with an arbitrary maturity.
So for  $t \in [0,T]$ and $x \ge 0$, we define
\[
v(t,x):= \sup_{I \in \hat \cA_x} J(t,x,I; \phi), \quad
J(t,x,I; \varphi):= \E_x\big[ - \sum_{\mathclap{0\le u<t}} e^{-\rho u}c(\Delta I_u) +
e^{-\rho t} \varphi(X^\alpha_t)
\1_{\{ \theta^\alpha \ge t\}}\big],
\]
where  as before $\alpha=(x,I)$,
$X^\alpha_{s}= x+ \mu s + \sigma W_s + I_s$, and
$\theta^\alpha= \inf \{ s >0 : X^\alpha_{s}<0 \}$.  It is
clear that $v(T,x)= \cOp \phi (x)$.
\begin{definition}
\label{d.space}
Let $\cC^T_A$  be the set of all  functions
$u:\overline{Q} \to \R$
such that
the map $x \in \Rnneg \mapsto u(t,x)-e^{-\rho t} x$ is
non-decreasing
and for every $(t,x) \in \overline{Q}$,
\[
e^{-\rho t}x\le u(t,x) \le
e^{-\rho t}(x+A+\mu t).
\]
\end{definition}
Note that the monotonicity of $u(t,x)-e^{-\rho t} x$ is
equivalent to $\partial_x u(t,x) \ge e^{-\rho t }$
in the distributional sense.

\begin{theorem}\label{t.cont}
For $\phi \in \cC_A$, the value function $v$
is the unique function in the space
$v\in \cinfty \cap \ccont \cap \cC^T_A$
satisfying,
\begin{gather}
\label{e.pde}
\rho v(t,x) + (\partial_t - \mu \partial_x- \frac12 \sigma^2 \partial_{xx})v(t,x) =0,
\quad (t,x) \in Q, \\
\label{e.pb}
v(t,0)= (\cI(v(t,\cdot))(0))^+=\sup_{\zeta>0} \left( v(t,\zeta)- c(\zeta)\right)^{\mathrlap{+}},
\quad t \in [0,T], \\
\label{e.terminal}
v(0,x)=\phi(x), \quad x\ge 0.
\end{gather}
\end{theorem}

We prove the existence of a solution by an iterative construction that converges to the solution.
To prove that the limit is indeed the solution, we first establish a number of regularity
results for the constructed sequence.
The proof of Theorem \ref{t.cont} is given later in this section, after these
 constructions and intermediate results.
The following definition simplifies the presentation.

\begin{definition}
\label{d.modulus}
A {\em{modulus}} (of continuity) is an increasing,
function $m: \Rnneg \to \Rnneg$, continuous at zero, and
with $m(0)=0$.
\end{definition}

Since $\phi \in \cC_A$, $0 \le \varphi:=\phi(x)-x \le A$ and $\varphi$
is increasing.  Therefore, $\varphi$ and consequently $\phi$ are
uniformly continuous.  Consequently, there exists a modulus
$m_\phi$ such that $|\phi(x)-\phi(y)| \le m_\phi(|x-y|)$
for every $x,y \ge 0$. Also, since $|\phi(x)-\phi(y)| \le A$
for every $x,y \ge 0$, we take $m_\phi \le A$.
For future reference, we set
\[
    m_\phi^*(t):= \E [m_\phi(\mu t + \sigma |W_t|)],\quad t \ge 0.
\]
It is clear that $m_\phi^*$ is also a modulus.

Let $\vzero$ be the solution of \eqref{e.pde}, \eqref{e.terminal}
with $\vzero(t,0)=0$ for all $t \in [0,T]$. By standard parabolic regularity
theory $\vzero \in \cinfty \cap \cC(\overline{Q}\setminus \{(0,0)\})$.
Let $X^x_s:= X_s^{(x,0)}$ and $\theta^x:= \theta^{(x,0)}$.  Then,
\[
\vzero(t,x)= \E_x \big[ e^{-\rho t} \phi(X^x_t) \1_{\{ \theta^x \ge t\}}\big],
\quad (t,x) \in \overline{Q}.
\]

\begin{lemma}
\label{l.zero}
 $\vzero \in \cinfty \cap \cC(\overline{Q}\setminus \{(0,0)\})\cap \cC^T_A$ and
 for every $\zeta_*>0$, there exits a modulus $m_0(\cdot; \zeta_*)$
such that
\[
\big| e^{\rho t }\vzero(t,x) -\phi(x)\big| \le m_0(t;\zeta_*), \quad
0\le t \le T, x \ge \zeta_*.
\]
In particular, the map $t \mapsto (\cI(\vzero(t,\cdot))(0))^+$
is continuous on $[0,T]$.
\end{lemma}
\begin{proof}
    Set $\vbzero(t,x):= e^{\rho t }\vzero(t,x)$.  Then, $\vbzero$ solves
\begin{equation}
\label{e.vb}
\cOp\vbzero(t,x):= (\partial_t -\mu \partial_x- \frac12 \sigma^2 \partial_{xx})\vbzero(t,x) =0,
\quad (t,x) \in Q.
\end{equation}

{\em{Step 1.}}
Set $w(t,x):= x+A+\mu t$.
Then, $w$ is a solution of \eqref{e.vb} and
as $\phi \in \cC_A$, $w(0,x) \ge \phi(x)=\vbzero(0,x)$.
Moreover, $w(t,0) \ge 0=\vbzero(t,0)$.  Then,
by maximum principle, $w \ge \vbzero$
on $\overline{Q}$. Similarly,  $u(t,x):=x$ is a
sub-solution of \eqref{e.vb} and $u(0,x) \le \phi(x)=\vbzero(0,x)$,
$u(t,0)=0=\vbzero(t,0)$.  Again by maximum
principle, $u \le \vbzero$ on $\overline{Q}$.

For $h>0$, set $\bar{w}(t,x):= \vbzero(t,x+h)-\vbzero(t,x)-h$.
Then, $\bar{w}(t,0)=\vbzero(t,h)-h = (\vbzero -u)(t,h) \ge 0$.
Also, since $\phi \in \cC_A$, $\bar{w}(0,x)= \phi(x+h)-\phi(x)-h \ge 0$.
Since $\bar{w}$ solves \eqref{e.vb},
we conclude that $\bar{w} \ge 0$
on $\overline{Q}$.  Hence, the map
$x \in \Rnneg \mapsto \vbzero(t,x) -x $ is increasing.

{\em{Step 2.}}
By the representation of $\vzero$,
\begin{align*}
|\vbzero(t,x)-\phi(x)|&
 \le \E_x [
|\phi(X^x_{t})-\phi(x)| \1_{\{ \theta^x \ge t\}}
+\phi(x) \1_{\{ \theta^x < t\}}]\\
&\le  \E_x[ m_\phi(X^x_{t}-x) \1_{\{ \theta^x \ge t\}}] +
\phi(x) \P (\theta^x < t).
\end{align*}
For any $\zeta_*>0$,
 $\hat{m}(t;\zeta_*):= \sup_{x \ge \zeta_*} \phi(x)
\P (\theta^x < t) $ is a modulus.  Hence,
\begin{align*}
|\vbzero(t,x)-\phi(x)|&
\le  \E_x[ m_\phi(\mu t + \sigma |W_{t}|) ] +
\phi(x) \P (\theta^x< t)\\
& \le m^*_\phi(t) +
\hat{m}(t;\zeta_*)=:m_0(t;\zeta_*),
\quad \forall\, x \ge \zeta_*.
\end{align*}
As both $m^*_\phi$ and $\hat{m}(\cdot;\zeta_*)$
are moduli, so is $m_0(\cdot;\zeta_*)$.

{\em{Step 3.}}  Since $\vzero(t,x) \le e^{-\rho t}(x+A+\mu t)$
and $c(\zeta) \ge \zeta$, for every $h > 0$ there is
$\zeta_h>0$ such that
\[
    \cI(\vzero(t,0))(0) = \sup_{0<\zeta \le \zeta_h} \vzero(t,\zeta) -c(\zeta),
    \quad \forall t \in [h,T].
\]
Because $\vzero$ is uniformly continuous on $[h,T]\times [0,\zeta_h]$
for any $h>0$,
we conclude that $t \mapsto \cI(\vzero(t,\cdot))(0)$ is
continuous on $(0,T]$.

{\textit{Step 4.}}
Next we prove the continuity at the origin.
Indeed, for any $t>0$ and $\zeta>0$,
$\vzero(t,\zeta)
\le \E_{\zeta} [ e^{-\rho t} \phi((X^{\zeta}_{t})^+)]
\le \E[ \phi(\zeta+\mu t+\sigma |W_{t}|)]$
and  $\phi(0) \ge \phi(\zeta ) +c(\zeta)$. Hence,
\[
\vzero(t,\zeta)-c(\zeta)-\phi(0)
\le  \E[ \phi(\zeta+\mu t+\sigma |W_{t}|)] -\phi(\zeta)
\le \E [  m_\phi(\mu t+\sigma |W_{t}|)]
 =m_\phi^*(t).
\]
Therefore,
$\cI(\vzero(t,\cdot))(0) =\sup_{\zeta >0} (\vzero(t,\zeta)-c(\zeta))
\le \phi(0)+m_\phi^*(t)$.
Since $\phi \ge 0$, we conclude that
$\limsup_{t\downarrow 0} (\cI(\vzero(t,\cdot))(0))^+ \le \phi(0)$.

{\em{Step 5.}}  Suppose $\phi(0)>0$. Then, $\phi(0)=\cI(\phi)(0)$.
For $\eps>0$, choose $\zeta_\eps >0$ such that
$\phi(0) \le \phi(\zeta_\eps)-c(\zeta_\eps)+ \eps$.
Thus,
\[
\liminf_{t \downarrow 0} \cI(\vzero(t,\cdot))(0)
\ge \lim_{t \downarrow 0} \vzero(t,\zeta_\eps)-c(\zeta_\eps)
= \phi(\zeta_\eps)-c(\zeta_\eps) \ge \phi(0)-\eps.
\]
The above estimate and Step 4 imply that
$\lim_{t\downarrow 0} (\cI(\vzero(t,\cdot))(0))^+=\phi(0)$.
\end{proof}

The following result is the iterative procedure.
\begin{proposition}
\label{p.iterate}
For $n\ge 1$, there are $\vn \in \cinfty \cap \ccont \cap \cC^T_A$
solving  the parabolic
equation \eqref{e.pde}, terminal data \eqref{e.terminal}, and
 the lateral boundary condition,
 \begin{equation}
 \label{e.boundary}
 \vn(t,0)=(\cI(v_{n-1}(t,\cdot))(0))^+, \quad t \in [0,T].
 \end{equation}
The unique solution satisfies
$\vn(t,x) \ge \cI(v_{n-1}(t,\cdot))(x)$ and
the representation
\begin{equation}
\label{e.fk}
\vn(t,x)= \E_x \big[ e^{-\rho t } \phi(X_t) \1_{\{ \theta^x \ge t\}}
+ e^{-\rho \theta^x} \vn(t-\theta^x,0) \1_{\{ \theta^x< t\}} \big],\quad
n=0,1,\ldots.
\end{equation}
Moreover for every $\zeta_*>0$, there exits a
modulus $m(\cdot; \zeta_*)$
such that
\[
\sup_{n\ge 0}\, \sup_{x \ge \zeta_*}\, \big| e^{\rho t}\vn(t,x) -\phi(x)\big| \le m(t;\zeta_*), \quad
\forall  t \in [0,T].
\]

\end{proposition}
\begin{proof}
We complete the proof in several steps.

{\em{Step 1.}}
In  view of the Lemma \ref{l.zero}, $\cI(\vzero(t,\cdot))(0)^+$
is continuous.  Then, by parabolic regularity theory,
there exists a unique
function $v_1 \in \cinfty \cap \ccont$ solving  the parabolic
equation \eqref{e.pde}, terminal data \eqref{e.terminal}, and
the lateral boundary condition $v_1(t,0)=\cI(\vzero(t,\cdot))(0)^+$.
The representation \eqref{e.fk} for $n=1$ follows
from the regularity of $v_1$.  As in Step 1 of the proof of
Lemma \ref{l.zero}, set $w(t,x):= x+A+ \mu t$.
Then, for every $\zeta >0$,
\[
\vzero(t,\zeta) -c(\zeta) \le e^{-\rho t}(\zeta +A + \mu t) -c(\zeta)
\le e^{-\rho t}(A + \mu t).
\]
Hence, $v_1(t,0) = \cI(\vzero(t,\cdot))(0)^+ \le e^{-\rho t} w(t,0)$.
We now proceed exactly as  in Step 1 of Lemma \ref{l.zero},
to prove that $v_1 \in \cC^T_A$.

{\em{Step 2.}}
Suppose that for $n\ge 1$, $\vn \in \cinfty \cap \ccont \cap \cC^T_A$ and
 the representation \eqref{e.fk} holds for $n$.
Then, we proceed as in Step 2 of the proof of Lemma \ref{l.zero}
using the inequality $|e^{\rho t}\vn(t,0)-\phi(x)| \le x+A+\mu T$
to obtain the following for $t \in [0,T]$ and $x \ge \zeta_*$,
\begin{align*}
|e^{\rho t }\vn(t,x)-\phi(x)|&
\le  \E_x[ m_\phi(\mu t + \sigma |W_{t}|) ] +
\sup_{x \ge \zeta_*} (x+A+\mu T) \P (\theta^x < t)\\
& \le m_\phi^*(t) +
\sup_{x \ge \zeta_*} (x+A+\mu T) \P (\theta^x < t)=:m(t;\zeta_*).
\end{align*}
It is clear  that $m(\cdot;\zeta_*)$
is a modulus.

{\em{Step 3.}}
We follow the Steps 3, 4 and 5 of  Lemma \ref{l.zero}
{\em{mutadis mutandis}}  to
show that
$(\cI(\vn(t,\cdot))(0))^+$ is continuous
on $[0,T]$.

We now use the parabolic regularity theory
to conclude that there exists a unique solution $v_{n+1} \in \cinfty \cap \ccont$
solving \eqref{e.pde}, \eqref{e.terminal} and \eqref{e.boundary}.  Also
the representation \eqref{e.fk} for $n+1$ follows from the regularity.

{\em{Step 4.}}  Fix $\zeta>0$ and set $u(t,x):= v_{n-1}(t,x+\zeta)-c(\zeta)$.
We directly verify that $u(0,\cdot) \le \phi=\vn(0,\cdot)$
and $u(\cdot,0)\le \vn(\cdot,0)$.  As $u$ solves
\eqref{e.pde}, we conclude that $u \le \vn$ on $\overline{Q}$.
Since $\zeta>0$ is arbitrary, $\cI(v_{n-1}(t,\cdot))(x) \le \vn(t,x)$.
 Finally, the argument used in the first step
 to prove that $\vn[1] \in \cC^T_A$ applies
directly to show that $v_{n+1} \in \cC^T_A$.
Then, we complete the proof of this proposition by induction.
\end{proof}

Let $\hat \cA_x^n$ be the set of all $I \in \hat \cA_x$
that have at  most $n$ issuances.
Since $\vn$ is smooth and $\vn(t,x) \ge \cI(\vn[n-1](t,\cdot))(x)$,
by a standard verification theorem,
we obtain
the following representation of $\vn$,
\[
\vn(t,x) =\sup_{I \in \hat \cA_x^n} J(t,x,I; \phi).
\]

\begin{lemma}
\label{l.boundary}
There exist a modulus $\widehat m$ such that
for $s \in [0,T]$,
\[
\sup_{n \ge 1}
\big| e^{\rho t} \vn(t,x)- \phi(x)\big| \le \widehat{m}(t),
\quad \forall (t,x) \in \overline{Q}.
\]
\end{lemma}
\begin{proof}
Set $\vbn(t,x):= e^{\rho t}\vn(t,x)$.

Fix $t \in [0,T]$, $x \ge 0$.
On $\{\theta^\alpha <t \}$, $X^\alpha_{\theta^\alpha}=0$
and $\Phi(0) \ge 0$.  Therefore,
$\phi(X^{\alpha}_t) \1_{\{ \theta^\alpha \ge t\}}\le
\phi(X^{\alpha}_{\theta^\alpha \wedge t})$
and
$e^{\rho t} J(t,x,I;\phi) \le
 \E_x[ - \sum_{0\le u<t} e^{\rho (t-u) }c(\Delta I_u) +
\phi(X^{\alpha}_{\theta^\alpha \wedge t})]$.
As $c$ is super-additive and positive,
\[
\sum_{0\le u<t} e^{\rho (t-u) }c(\Delta I_u)
\ge \sum_{0\le u<t} c(\Delta I_{u})
\ge c(\sum_{0\le u<t}\Delta I_u)
= c(I_{\theta^\alpha \wedge t}).
\]
Moreover, because $\phi \ge \cI(\phi)$,
\[
\phi(X^\alpha_{\theta^\alpha \wedge t}) - c(I_{\theta^\alpha \wedge t})
\le \phi((X^\alpha_{\theta^\alpha \wedge t}- I_{\theta^\alpha \wedge t})^+)
\le \phi(x+\mu ({\theta^\alpha \wedge t})
+\sigma |W_{\theta^\alpha \wedge t}|).
\]
Recall that $m_\phi^*(t):= \E[ m_\phi(\mu t +\sigma |W_t|)]$.
We  combine
these inequalities to arrive at the following estimate,
\[
e^{\rho t} J(t,x,I;\phi)  -\phi(x)\le
\E_x\big[\phi(x+\mu ({\theta^\alpha \wedge t})
+\sigma |W_{\theta^\alpha \wedge t}|) -\phi(x)\big]
\le m_\phi^*(t).
\]
Therefore, $\vbn(t,x)- \phi(x) \le m_\phi^*(t)$.

We continue by using \eqref{e.fk} which is equivalent to
\[
\vbn(t,x)= \E_x \big[\phi(X_t) \1_{\{ \theta^x \ge t\}}
+\vbn(t-\theta^x,0) \1_{\{ \theta^x< t\}} \big].
\]
Since on $\{ \theta^x <t\}$, $\phi(X_{{\theta^x \wedge t}}^x)=\phi(0)$,
\begin{align*}
\phi(x) -\vbn(t,x) &= \phi(x)-
\E_x[
\phi(X^x_t)\1_{\{ \theta^x \ge t\}}+
 \vbn(t-\theta^x,0)\1_{\{ \theta^x<t\}}]\\
 &= \phi(x)-
\E_x[
\phi(X^x_{\theta^x\wedge t})+
(\vbn(t-\theta^x,0)- \phi(0))\1_{\{ \theta^x<t\}}]\\
&= \E_x[
(\phi(x) - \phi(X_{{\theta^x \wedge t}}^x))+
\big(\phi(0) - \vbn(t-\theta^x,0)\big)\1_{\{ \theta^x<t\}}].
\end{align*}
If $\phi(0)=0$, then
$\phi(x) -\vbn(t,x) \le \E_x [ |\phi(x) - \phi(X^x_{\theta^x \wedge t})|] \le m_\phi^*(t)$.
Suppose $\phi(0)>0$.
Then, $\phi(0)= \cI(\phi)(0)^+=\cI(\phi)(0)$.
Recall that $c(\zeta)=\lambda_f+(1+\lambda_p)\zeta$
with $\lambda_f>0$ and $\phi$ is continuous.  Hence, there is
$\zeta_*>0$ such that $\phi(0)=\cI(\phi)(0)= \sup\{ \phi(\zeta)-c(\zeta)
:\zeta \ge \zeta_*\}$.
For $\eps>0$, choose $\zeta_\eps \ge \zeta_*$ such that
$\phi(0) \le \phi(\zeta_\eps)-c(\zeta_\eps)+ \eps$.
Therefore,  we have
$\vn(t-\theta^x,0) \ge
v_{n-1}(t-\theta^x, \zeta_\eps)
-c(\zeta_\eps)$.  By Proposition \ref{p.iterate},
on $\{\theta^x<t\}$ we have
\[
\phi(0)-\vn(t-\theta^x,0)
\le \phi(\zeta_\eps) - v_{n-1}(t-\theta^x, \zeta_\eps) +\eps
\le m(t-\theta^x;\zeta_*)+\eps
\le m(t;\zeta_*)+\eps.
\]
Therefore,
\[
\phi(0)-\vbn(t-\theta^x,0) =e^{\rho t}(\phi(0)-\vn(t-\theta^x,0)) +(e^{\rho t}-1) \phi(0)
\le e^{\rho T} m(t;\zeta_*) + A(e^{\rho t}-1).
\]
Hence, in both cases
$\phi(x) -\vbn(t,x) \le m_\phi^*(t)+ [e^{\rho T} m(t;\zeta_*) + A(e^{\rho t}-1)]=: \widehat{m}(t)$.
\end{proof}

\begin{proposition}
\label{p.time}
There exists a  modulus $\tilde{m}(\cdot)$
such that
\[
\sup_{n \ge 1} \big| \vn(t,x) -  \vn(t+h,x) \big|
\le (1+x) \tilde{m}(h),
\quad \forall t \in [0,T-h], x \ge 0.
\]
\end{proposition}

\begin{proof}
As before, set
$\vbn(t,x):= e^{\rho t}\vn(t,x)$.

{\em{Step 1.}} Since $\vn \in \cC^T_A$,
there is a constant $c_0\ge1 $, independent of $n$,
such that
$0 \le \vn(t,x) \le c_0 (1+x)$.
Moreover, because $\pcost, \fcost>0$ and $\vn \in \cC^T_A$,
there is $\zeta^*>0 $ again independent of $n$,
such that for every $(t,x) \in \overline{Q}$,
\[
\cI(\vn(t,\cdot)(x)=\sup_{\zeta \in [0,\zeta^*]} \vn(t,x+\zeta) -c(\zeta).
\]

{\em{Step 2.}}   Fix $(t,x) \in \overline{Q}$, $h>0$ and
set $\tau:=\theta^x \wedge t$.
By Feynman--Kac formula, for any $u \geq t$,
$\vn(u,x) = \E_x[ e^{-\rho \tau} \vn(u-\tau,X^x_\tau)]$.
We use this identity with the choices $u=t$
and $u=t+h$,
which implies that
\[
\vn(t,x)-\vn(t+h,x)
=\E_x[ e^{-\rho \tau}( \vn(t-\tau,X^x_{\tau})
-\vn(t+h-\tau,X^x_{\tau}))].
\]
Separating into the two cases $\tau=t$
and $\tau=\theta^x<t$ and dropping the exponential factor, we obtain
\begin{equation}%
    \label{eqn:vmodulusAB}
\left|\vn(t,x)-\vn(t+h,x) \right|
\le  \E_x{\left [ A_n \1_{\{\theta^x \ge t\}} \\
+  B_n \1_{\{\theta^x \le t\}} \right]},
\end{equation}
where
\[
A_n := \left|\phi(X^x_t) -\vn(h,X^x_t)\right|,\quad
B_n:=  {\left|\vn(t-\theta^x,0)-\vn(t+h-\theta^x,0)\right|}.
\]

{\em{Step 3.}}  Let $c_0$ be as in Step 1
and set $c:= c_0 e^{\rho T}$.  Then
by Lemma \ref{l.boundary},
\[
A_n =  \left| [\phi(X^x_t)
-\vbn(h,X^x_t)] +
(e^{\rho h }-1) \vn(h,X^x_t) \right|
\le  \widehat{m}(h)+ c (1+X^x_t) \rho h.
\]
Hence, for any $t \in [0,T-h], x \ge 0$,
\begin{equation}
\label{e.A}
\E_x [ A_n \1_{\{\theta^x \ge t\}}]
\le
(\widehat{m}(h)+ c \rho h ) \P(\theta^x \ge t) +
c  \rho h\,  \E_x [ X^x_t  \1_{\{\theta^x \ge t\}}].
\end{equation}

{\em{Step 4.}}
We now establish a modulus bound for $B_n$.
Since $\vn$ is continuous,
\[
\tilde{m}_n(h):= \sup_{t \in [0,T-h], x \in [0,\zeta^*]}
\left| \vn(t,x)-\vn(t+h,x)\right|
\]
is a modulus for $n \geq 1$.
Then, for $n>1$,
in view of  the boundary condition \eqref{e.boundary},
\[
B_n \le \sup_{\zeta \in [0,\zeta^*]} \left| v_{n-1}(t-\theta^x,\zeta)-
v_{n-1}(t+h-\theta^x,\zeta) \right|
\le \tilde{m}_{n-1}(h).
\]
We show that $\tilde{m}_n$ is uniformly bounded by a modulus.
First, there exists $c^* > 0$ depending on $\zeta^*$ such that
for all $t \in [0,T-h]$ and $x \in [0,\zeta^*]$ we have
$\E_x [ X^x_t  \1_{\{\theta^x \ge t\}}]
\le c^* \P(\theta^x \ge t)$.
Therefore,
\[
\E_x [ A_n \1_{\{\theta^x \ge t\}}]
\le [\widehat{m}(h)+ c  (1+c^*) \rho h  ] \P(\theta^x \ge t), \quad
\forall \, t \in [0,T-h], x \in [0,\zeta^*].
\]
The combination of \eqref{eqn:vmodulusAB} and the above two estimates implies that
\begin{align*}
\tilde{m}_n(h) & \le [\widehat{m}(h)+ c (1+c^*) \rho h ] \P(\theta^x \ge t)
+  \tilde{m}_{n-1}(h) (1-\P(\theta^x \ge t))\\
& \le \sup_{\lambda \in [0,1]}
[\widehat{m}(h)+ c (1+c^*) \rho h] \lambda
+  \tilde{m}_{n-1}(h) (1-\lambda))\\
&=  \max\{ \widehat{m}(h)+ c  (1+c^*)\rho h \, ,\,
 \tilde{m}_{n-1}(h) \}, \quad n >1.
\end{align*}
By induction, we conclude that
\begin{equation}%
    \label{eqn:modulus_B}
B_n \leq \tilde{m}_n(h) \le \max\{ \widehat{m}(h)+ c (1+c^*)\rho h \, ,\,
 \tilde{m}_{1}(h) \} \eqfed m_B(h), \quad n \ge 1.
\end{equation}

{\em{Step 5.}}
Returning to \eqref{e.A},
$\E_x [ X^x_t  \1_{\{\theta^x \ge t\}}]
\le c (1+x)$ for all $x \geq 0$.
Thus, for all $x \geq 0$,
\begin{equation}%
    \label{eqn:modulus_A}
    \begin{aligned}
        \E_x [ A_n \1_{\{\theta^x \ge t\}}] &\leq
        (\widehat{m}(h)+ c \rho h ) \P(\theta^x \ge t) +
        c  \rho h  \E_x [ X^x_t  \1_{\{\theta^x \ge t\}}]\\
        &\le  (\widehat{m}(h)+ c \rho h )  +
        c  \rho h  [c(1+x)]\\
        &\le (\widehat{m}(h)+ c(1+c) \rho h ) (1+x)
        =:
        m_A(h) (1+x).
    \end{aligned}
\end{equation}
Plugging \eqref{eqn:modulus_B} and \eqref{eqn:modulus_A} into \eqref{eqn:vmodulusAB} thus yields
\[
\left|\vn(t,x)-\vn(t+h,x) \right|
\le m_A(h) (1+x) + m_B(h).
\]
\end{proof}

We can now complete the proof of Theorem~\ref{t.cont}.

\begin{proof}[Proof of Theorem~\ref{t.cont}.]
By their definitions, $v_1(t,0)\ge 0=\vzero(t,0)$ and
$v_1(T,\cdot)=\vzero(T,\cdot)=\phi$.
As both $v_1$ and $\vzero$ solve \eqref{e.pde} and $v_1\ge\vzero$ on the boundary,
$v_1 \ge \vzero$ on $\overline{Q}$.  Suppose that $\vn \ge v_{n-1}$ on $\overline{Q}$
for some $n\ge 1$.  Then,
\[
v_{n+1}(t,0) = \cI(\vn(t,\cdot))(0)^+\ge  \cI(v_{n-1}(t,\cdot))(0)^+=\vn(t,0), \quad t \in [0,T].
\]
We now argue as in the case of $v_1$
to conclude that $v_{n+1} \ge \vn$ on $\overline{Q}$.
By induction, we conclude that $\vn$ is increasing in $n$.
Suppose that the sequence $\{\vn\}_n$
is uniformly locally continuous. Then, $\{\vn\}_n$ converges locally uniformly to $v \in \ccont$.
As $\{\vn\}_n$ all solve \eqref{e.pde}, so does $v$, and by parabolic regularity,
$v \in \cinfty$. Moreover, local uniform convergence of $\{\vn\}_n$ implies that
$v \in \ccont \cap \cC^T_A$.  Since $\vn(t,x) \ge \cI(v_{n-1}(t,\cdot))(x)$ for every $(t,x)\in \overline{Q}$,
we conclude that $v(t,x) \ge \cI(v(t,\cdot))(x)$ as well. The regularity of $v$ together
with \eqref{e.pde} and the boundary conditions allow us to prove by
standard verification arguments that $v$ is the value function.
Hence, it suffices the sequence $\{\vn\}_n$
is locally uniformly  continuous.  In the remainder
of this proof we establish this property.

By Proposition~\ref{p.time},
$\vn(t,0)=\cI(v_{n-1}(t,\cdot))(0)^+$ is uniformly continuous, i.e.,
$|\vn(t,0)-\vn(t+h,0)| \le \tilde{m}(h)$.
Fix $(t,x) \in \overline{Q}$. Since $\vn(0,0)=\phi(0)$,
by \eqref{e.fk},
\begin{align*}
\vn(t,h) -\vn(t,0) &=  \E_h [e^{-\rho t}
\phi(X^h_t) \1_{\{\theta^h\ge t\}}
+e^{-\rho \theta^h}
 \vn(t-\theta^h,0)\1_{\{\theta^h < t\}}] -\vn(t,0)\\
&\le   \E_h [ (\phi(X^h_{\theta^h \wedge t}) -\phi(0))
\1_{\{\theta^h\ge t\}}
+ (\vn((t-\theta^h)^+,0)-\vn(t,0))]\\
&\le \E_h[ m_\phi(h+\mu t + \sigma|W_t|) 1_{\{\theta^h\ge t\}}
+ \tilde{m}(\theta^h)].
\end{align*}
Note that
$m(h):=\sup_{t \in[0,T]} \E_h[ m_\phi(h+\mu t + \sigma|W_t|) 1_{\{\theta^h\ge t\}}]
+ \E_h[\tilde{m}(\theta^x)]$ is a modulus.
Moreover, by \eqref{e.fk},
\begin{align*}
\vn(t,x+h) -\vn(t,x) &=
\E_x \big[ \begin{aligned}[t]
        &e^{-\rho t } (\phi(X_t+h) -\phi(0,X_t)) \1_{\{ \theta^x \ge t\}}\\
        & + e^{-\rho \theta^x} (\vn(t-\theta^x,h)-\vn(t-\theta^x,0)) \1_{\{ \theta^x< t\}}
    \big]\end{aligned}\\
& \le m_\phi(h)+ m(h).
\end{align*}
This together with
Proposition \ref{p.time} imply that the sequence $\{\vn\}_n$
is locally uniformly continuous.

Finally, let $u$ be any $\ctwo \cap \ccont \cap \cC^T_A$ solution.
Then there exists a $\delta > 0$ such that in a $\delta$-neighborhood of $[0,T] \times \{ 0 \}$ it holds that $u < \fcost$.
Hence, the maximizers
\[
    \zeta^*_\cdot \in \argmax_{\zeta \geq 0} \bigl(u(\cdot, y) - u(\cdot, 0) - \fcost - (1 + \pcost) y\bigr) \1[y > 0]
\]
generate an issuance policy $I^*$ by issuing $\zeta^*_\tau$ at a time $\tau$ when the reserves reach zero.
Moreover, $\Delta I^* \geq \delta$ whenever non-zero.
Hence, the jump times of $I^*$ do not have cluster points $P$-a.s., and $I^*$ is thus an admissible control.
The usual verification arguments then show that $u$ is the value function.
\end{proof}

\subsection{Numerical results}
\label{sec:results_1D}
To compute the value function, the two operators $\dOp$ in \eqref{eqn:dOp1D} and $\cOp$ in \eqref{eqn:cOp1D} need to be implemented.
The former is straight-forward to implement, but the latter requires a bit more work.

For a model without equity issuance, $\cOp$ can for instance easily be implemented by means of Monte Carlo simulations.
This is particularly convenient for cash flow processes without diffusion, like Cramér--Lundberg model.
The reason for this is that to evaluate the indicator function in \eqref{eqn:cOp1D}, a test for ruin only has to be made at the time of a jump.
On the other hand, in a diffusion model, this has to be estimated by making increasingly smaller time steps.

Fortunately, also with equity issuance, alternative methods can be employed to solve the problem.
We use the PDE representation \eqref{e.pde}, \eqref{e.terminal},
\eqref{e.boundary} and opt for the semi-Lagrangian method presented in \cite{azimzadeh2017convergence}.
Since we wish to compute the solution on a bounded domain, an artificial boundary condition has to
also be specified.
At any time point, any additional inflow of cash at the upper boundary is paid out as dividends at the next opportunity, provided the reserves do not fall below the dividend barrier.
As the computational domain is chosen larger, it is therefore increasingly unlikely that additional cash is not paid out.
Hence, if the domain is chosen sufficiently large, the present value of $\Delta x$ at the boundary is its discounted value $e^{-\discRate (T-t)} \Delta x$.
The boundary condition $v_x(x, t) = e^{-\discRate (T-t)}$ is therefore a good approximation.

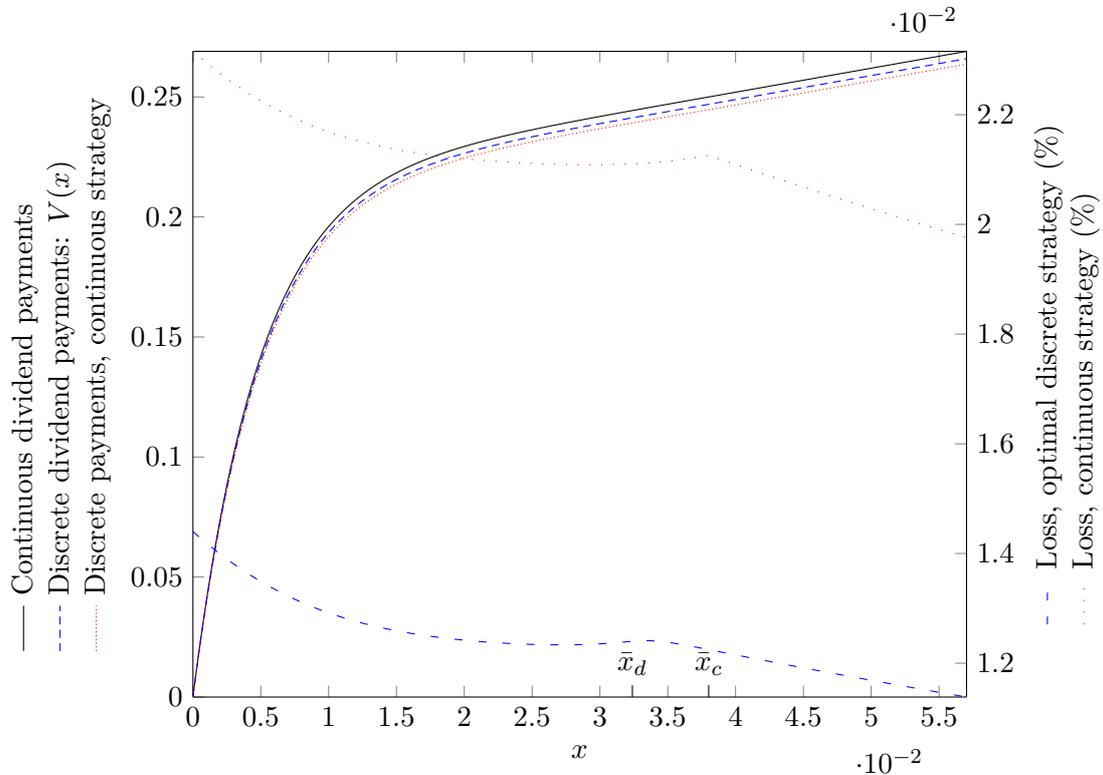
\begin{figure}
	\tikzset{external/export next=false}
	\begin{center}
		\newcommand{\axiswidth}{0.80\textwidth}
		\begin{tikzpicture}
			\begin{axis}[axis y line*=left, xlabel=$x$, ylabel style = {align=left}, ylabel={\ref{plot:noissuance1D:JBS} Continuous dividend payments \\ \ref{plot:noissuance1D:V} Discrete dividend payments: $V(x)$ \\ \ref{plot:noissuance1D:Vwrong} Discrete payments, continuous strategy}, width=\axiswidth, enlargelimits=false, no markers, legend pos=south east, legend cell align=left, yticklabel style={/pgf/number format/fixed}, legend style={at={(0.97,0.4)},anchor=south east}]
				\addplot[black] table [x=x, y=JBS, col sep=comma] {Figures/calibrated.csv};
				\label{plot:noissuance1D:JBS}
				\addplot[densely dashed, blue] table [x=x, y=V, col sep=comma] {Figures/calibrated.csv};
				\label{plot:noissuance1D:V}
				\addplot[densely dotted, red] table [x=x, y=Vwrong, col sep=comma] {Figures/calibrated.csv};
				\label{plot:noissuance1D:Vwrong}
				\draw (0.03239073944241952, 0) -- (0.03239073944241952, 0.005) node[above] {$\bar{x}_d$};
				\draw (0.03801729981504639, 0) -- (0.03801729981504639, 0.005) node[above] {$\bar{x}_c$};
			\end{axis}
			\begin{axis}[axis y line*=right, axis x line=none, ylabel style = {align=left}, ylabel={\ref{plot:noissuance1D:loss} Loss, optimal discrete strategy (\%) %
				\\ \ref{plot:noissuance1D:losswrong} Loss, continuous strategy (\%)},%
				width=\axiswidth, enlargelimits=false, no markers, legend pos=south east, legend cell align=left, yticklabel style={/pgf/number format/fixed}]
				\addplot[loosely dashed,blue] table [x=x, y=loss, col sep=comma] {Figures/calibrated.csv};
				\label{plot:noissuance1D:loss}
				\addplot[loosely dotted,red] table [x=x, y=losswrong, col sep=comma] {Figures/calibrated.csv};
				\label{plot:noissuance1D:losswrong}
			\end{axis}
		\end{tikzpicture}
	\end{center}
	\caption{\label{fig:noissuance1D}\textbf{Value functions and dividend policies without equity issuance.}
	On the left axis are plots of the value function $V$ in the discrete model (blue, dashed), the value function for the problem with continuous (singular) dividend payments (black, solid), and the value obtained from (suboptimally) using the optimal continuous strategy in the discrete model (red, dotted). On the right axis are the losses in percent due to discretization of dividend payments, relative to the continuous model. The two lines denote the losses using the optimal discrete strategy as well as the (suboptimal) continuous strategy. The cash flow is given by $C_t = \mu t + \xsigma \xW_t$ and the parameters values are $\discRate=0.04$, $T=1$, $\sigma=0.01$, and $\mu=0.01$. The values $\bar{x}_d$ and $\bar{x}_c$ at the bottom are the dividend barriers in the discrete and continuous models respectively. The difference corresponds to 14\% lower reserves in the discrete model.}
\end{figure}

With means for calculating both operators $\dOp$ and $\cOp$, we may proceed to iteratively apply $\iterOp$ to any arbitrary initial function.
Our choice of parameters for the computations in this section comes from \cite{peura_optimal_2006} and are listen in Figure~\ref{fig:noissuance1D}.
For our purposes, we consider all parameters to be in fractions of their so-called regulatory risk-weighted assets.

The results without equity issuance is presented in Figure~\ref{fig:noissuance1D}.
The loss due to discretization of dividend payments quickly falls to a level below 1.4\%.
The primary impression of the result is that the loss is relatively small.
Note that for larger values of $x$, the absolute loss stays constant, so the relative loss decays as the value function increases linearly.
Although the change in the value function is not very large, the dividend barrier moves considerably, decreasing by 14\% (slightly more than 0.005 units) in Figure~\ref{fig:noissuance1D}.
We attribute this mainly to paying out some of the expected income during the next period.
Note, however, that due to the need of keeping a buffer, only a bit more than half of the expected cash flow is paid out in advance.

One important aspect of discretization of dividends is that the use of the continuous time optimal dividend threshold in the discrete problem induces further losses, since it is no longer optimal in the discrete model.
To shed some light on the effect of using the wrong policy in this way, Figure~\ref{fig:noissuance1D} also shows the value function resulting from using the continuous time dividend barrier of the discrete dividend payments (the smallest of the three functions).
We also plot the relative losses in comparison to the continuous dividend model.
For the parameters in the figure, we observe that using the wrong policy adds a bit more than 0.8 percentage points to the losses.

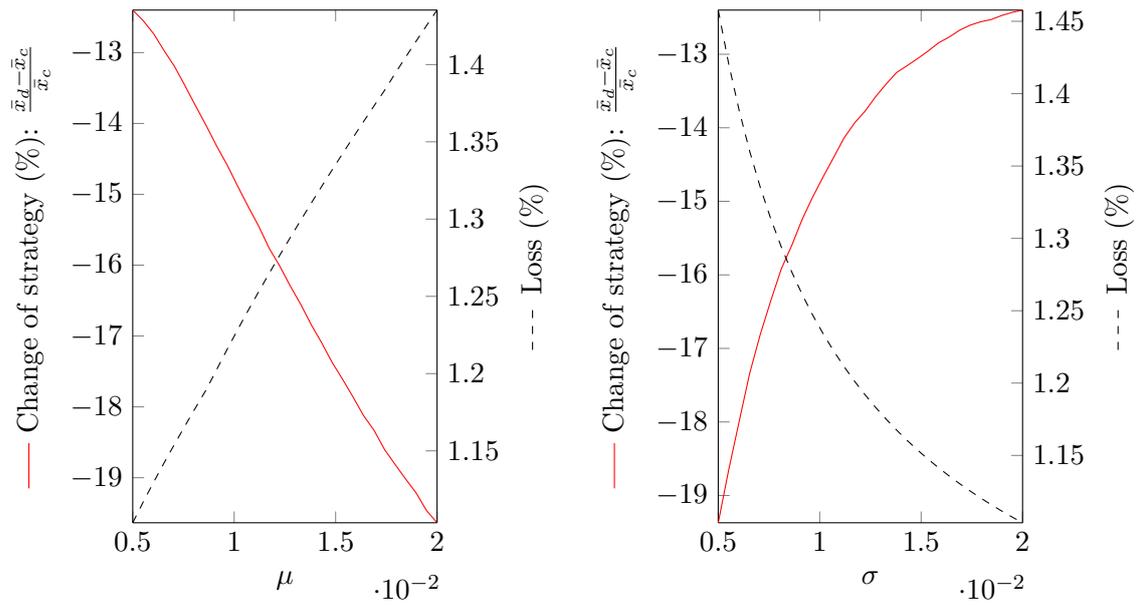
\begin{figure}
	\tikzset{external/export next=false}
	\begin{center}
		\newlength{\axiswidth}
		\setlength{\axiswidth}{0.38\textwidth}
		\newlength{\axisheight}
		\setlength{\axisheight}{1.5\axiswidth}
		\begin{tikzpicture}
			\pgfplotsset{every axis/.append style={height=\axisheight, width=\axiswidth,}}
			\matrix[column sep=1em]{
				\begin{axis}[axis y line*=left, xlabel=$\mu$, ylabel={\ref{plot:mu:xbarchange} Change of strategy (\%): $\frac{\bar{x}_d - \bar{x}_c}{\bar{x}_c}$}, enlargelimits=false, no markers, yticklabel style={/pgf/number format/fixed}]
				\addplot[red] table [x=mu, y=xbarchange, col sep=comma] {Figures/noissuancelossMu.csv};
				\label{plot:mu:xbarchange}
			\end{axis}
			\begin{axis}[axis y line*=right, axis x line=none, ylabel={\ref{plot:mu:loss} Loss (\%)},%
				enlargelimits=false, no markers, legend pos=south east, legend cell align=left, yticklabel style={/pgf/number format/fixed}, legend cell align=left, legend style={at={(0.15,0.97)},anchor=north west}]
				\addplot[dashed,black] table [x=mu, y=loss, col sep=comma] {Figures/noissuancelossMu.csv};
				\label{plot:mu:loss}
			\end{axis}
			&
			\begin{axis}[axis y line*=left, xlabel=$\sigma$, ylabel={\ref{plot:sigma:xbarchange} Change of strategy (\%): $\frac{\bar{x}_d - \bar{x}_c}{\bar{x}_c}$}, enlargelimits=false, no markers, yticklabel style={/pgf/number format/fixed}]
				\addplot[red] table [x=sigma, y=xbarchange, col sep=comma] {Figures/noissuancelossSigma.csv};
				\label{plot:sigma:xbarchange}
			\end{axis}
			\begin{axis}[axis y line*=right, axis x line=none, ylabel={\ref{plot:sigma:loss} Loss (\%)},%
				enlargelimits=false, no markers, legend pos=south east, legend cell align=left, yticklabel style={/pgf/number format/fixed}, legend cell align=left, legend style={at={(0.15,0.97)},anchor=north west}]
				\addplot[dashed,black] table [x=sigma, y=loss, col sep=comma] {Figures/noissuancelossSigma.csv};
				\label{plot:sigma:loss}
			\end{axis}
			\\
		};
		\end{tikzpicture}
	\end{center}
	\caption{\label{fig:musigma} \textbf{Effect of the parameters $\mu$ and $\sigma$ without issuance.}
		The change in strategy in terms of the relative distance between the continuous dividend strategy and the discrete one. The loss is evaluated at the optimal dividend barrier for the continuous problem. The fixed parameters are the same as in Figure~\ref{fig:noissuance1D}.}
\end{figure}

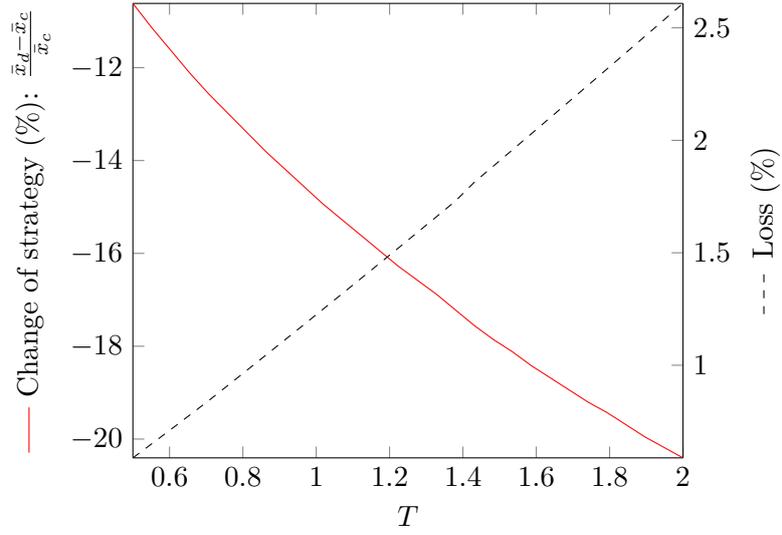
\begin{figure}
	\begin{center}
		\tikzset{external/export next=false}
		\begin{tikzpicture}[scale=1, transform shape]
			\setlength{\axiswidth}{0.6\textwidth}
			\begin{axis}[axis y line*=left, xlabel=$T$, ylabel={\ref{plot:T:xbarchange} Change of strategy (\%): $\frac{\bar{x}_d - \bar{x}_c}{\bar{x}_c}$}, width=\axiswidth, enlargelimits=false, no markers, yticklabel style={/pgf/number format/fixed}]
				\addplot[red] table [x=T, y=xbarchange, col sep=comma] {Figures/noissuancelossT.csv};
				\label{plot:T:xbarchange}
			\end{axis}
			\begin{axis}[axis y line*=right, axis x line=none, ylabel={\ref{plot:T:loss} Loss (\%)},%
				width=\axiswidth, enlargelimits=false, no markers, legend pos=south east, legend cell align=left, yticklabel style={/pgf/number format/fixed}, legend cell align=left, legend style={at={(0.15,0.97)},anchor=north west}]
				\addplot[dashed,black] table [x=T, y=loss, col sep=comma] {Figures/noissuancelossT.csv};
				\label{plot:T:loss}
			\end{axis}
		\end{tikzpicture}
	\end{center}
	\caption{\label{fig:T} \textbf{Effect of the parameter $T$ without issuance.}
		The change in strategy measures the relative distance between the continuous dividend strategy and the discrete one. The loss is evaluated at the optimal dividend barrier for the continuous problem. The fixed parameters are the same as in Figure~\ref{fig:noissuance1D}.}
\end{figure}

Figures \ref{fig:musigma} and \ref{fig:T} both show the effect of varying some of the parameters.
The loss comparisons are all made at the optimal barrier of the continuous continuous model, $\bar{x}_c$.
The rationale for this is that it is the level of reserves of a healthy firm.
Changes to $\mu$ and $\sigma$ that increase the value of the continuous model also raises the \emph{relative} loss from discretizing the dividend strategy.
We also observe that for all parameter values in the ranges considered, there is a significant shift in the optimal strategy.
In Figure~\ref{fig:T} we solve the problem for different values of $T$.
This is the only parameter that does not affect the continuous time problem.
As expected, the size of $T$ has a strong impact on the losses.
The figure suggests that the loss in the value function is low for quarterly dividend payments, but still the dividend strategy is quite different.

\begin{figure}
	\begin{center}
		\begin{tikzpicture}
	\pgfplotsset{every axis/.append style={width=0.5\textwidth, colormap/viridis}}
	\matrix[row sep=2em]{
	\begin{axis}[%
		view={-46}{26},
		xmajorgrids,
		ymajorgrids,
		zmajorgrids,
		enlargelimits=false,
		xtick={0,0.025,0.05},
		xlabel=$x$,
		ylabel=$t$,
		every x tick scale label/.style={
			at={(xticklabel cs:0.8,0pt)},
			anchor=near xticklabel,inner sep=0pt},
	  ]

	  \addplot3[%
		  surf,
		  shader=interp,
		  samples=120,
		  mesh/cols=51]
	  table[col sep=comma] {Figures/surf.csv};
	  \addplot3[draw=white!100] table[col sep=comma] {Figures/issuance1D.csv};

	\end{axis}
	&
	\begin{axis}[%
		view={0}{90},
		enlargelimits=false,
		xtick={0,0.025,0.05},
		xlabel=$x$,
		ylabel=$t$,
		every x tick scale label/.style={
			at={(xticklabel cs:0.7,-5pt)},
			anchor=near xticklabel,inner sep=0pt},
	  ]

	  \addplot3[%
		  surf,
		  shader=interp,
		  samples=120,
		  mesh/cols=51]
	  table[col sep=comma] {Figures/surf.csv};
	  \addplot3[draw=white!100] table[col sep=comma] {Figures/issuance1D.csv};

	\end{axis}
\\};
\end{tikzpicture}%
		\vspace*{-2.5em}
		\caption{\label{fig:issuance1D} \textbf{Surface plots of the value function $v(x, t)$ with issuance, and the optimal issuance target for $t \in [0,1)$.}
			Issuance only occurs at the boundary $x=0$, and the issuance target is presented as the white line on the surface. The issuance costs are $\pcost = 0$, $\fcost = 0.0025$, and the remaining parameters are the same as in Figure~\ref{fig:noissuance1D}.}
	\end{center}
\end{figure}
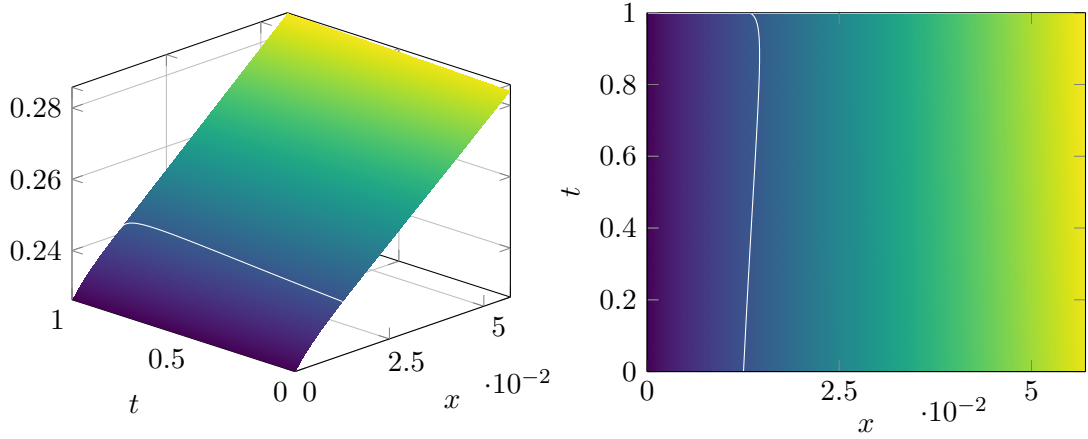

Figure~\ref{fig:issuance1D} shows the value function for the model with issuance.
As expected for issuance costs independent of $t$ and $x$, issuance only occurs at the boundary.
Note that since $\pcost=0$, the optimal issuance target\footnote{Because $\pcost = 0$, the target is not unique at time points coinciding with dividend payments, since excessive issuance can be offset by dividend payments at no cost. At these points we consider the optimal issuance target to be the smallest optimizer.} coincides with the dividend barrier, in this case roughly 0.0125.
We observe that the size of issued equity grows as time passes, with the exception of the period right before the time of dividend payment where it drops to its initial value.

\section{Discrete dividends with random profitability}
\label{sec:dividends2D}
Instead of the constant drift considered in Section~\ref{sec:results_1D}, one could consider the drift---the \emph{profitability}---to be described by another random process.
Suppose that the cash flow $\dif{C_t} = \mu_t \dif{t} + \xsigma \dif{\xW_t}$
depends on some profitability process
$\process{\mu}$.
The net cash reserves $X = \process{X}$ depends on the initial
cash reserve $x \ge 0$, the initial profitability $\mu \in \R$ and
the control process $(I,L)$.  With abuse of notation
we use $\nu=(x,\mu,L,I)$ to denote these dependences
and
\[
X^\nu_t = x + C^\mu_t -L_t + I_t, \quad
C^\mu_u=\int_0^t \mu_u \dif{u} +\xsigma \xW_t.
\]
Let $\filtration = \process{\filtration}$ be the filtration generated by $(C^\mu, \mu)$.
Again, we restrict dividends to be fully covered by reserves of the firm,
i.e., $\Delta L_t \leq X^{L,I}_{t-}$, and the ruin time is given by
$\ruinTime^\nu = \inf \{ t > 0 : X^\nu_t < 0 \}$.

Just as before, the aim of the firm is to maximize
the discounted value of dividends net of equity issuance
and the value function $V(x,\mu)$ and $J(\nu)$
are given as before.  The main difference
is the dependence on the initial profitability.

\subsection{Periodization and numerical convergence}
We define the operators $\cOp$, $\dOp$ and $\iterOp$ exactly as before.
In the proof of Theorem~\ref{thm:convergence_1D}, the equity
issuance poses no extra obstacle.
However, for the sake of simplifying the exposition,
we present the results without equity issuance.
In particular,
we do not prove the regularity of the value function
as we did in the previous section.  Instead
we work within the class of universally measurable functions.

We make the following assumption on $C^\mu$ and $\process{\mu}$
and it  ensures that the effect of random profitability is sufficiently well behaved.
In particular, it restricts the profitability process from having too strong growth.
\begin{assumption}
    \label{ass:alpha}
    There exists an $\alpha : \mathbb{R} \to [1, \infty)$ so that for all $\mu$ we have
    \begin{enumerate}
		\item $\E_{x,\mu} [ (x + C_{T-}^\mu)^+ ] \leq x + A \alpha(\mu)$, for some $A \geq 0$;
		\item $\E_{x,\mu} [ \alpha(\mu_{T-}) ] \leq e^{\discRate T /2} \alpha(\mu)$.
    \end{enumerate}
\end{assumption}
Now we can give the following result.
\begin{theorem}
	There exists a metric space $(\mathcal{X}_\alpha, \alphaMetric)$ such that the
	operator $\iterOp$ maps $\mathcal{X}_\alpha$ into itself and is a strict contraction. Moreover,
	the value function is the unique fixed-point  of $\iterOp$.
\end{theorem}

\begin{proof}
    We prove the statements for the following subspace of universally measurable functions:
    \[
		\mathcal{X}_\alpha \eqdef \left\{ x \leq \phi(x, \mu) \leq x + A_\phi \alpha(\mu) \text{ for some } A_\phi \right\}
	\]
	with metric
	\[
		\alphaMetric( \phi, \psi ) \eqdef \sup_{x \geq 0, \mu \in \mathbb{R}} \frac{|\phi(x, \mu) - \psi(x,\mu)|}{\alpha(\mu)}.
	\]
	Note that this implies that $|\phi(x,\mu) - \psi(x,\mu)| \leq \alphaMetric( \phi, \psi ) \alpha(\mu)$.

	Then, for $\phi \in \mathcal{X}_\alpha$,
    \begin{align*}
		e^{\discRate T} \cOp \phi(x,\mu) &\leq \E_{x,\mu} [ (X_{T-} + A_\phi \alpha(\mu_{T-})) \1_{\{\theta \geq T\}}] \\
		& \leq \E_{x,\mu}[ (x + C^\mu_{T-})^+] + A_\phi \E_{x,\mu}[\alpha(\mu_{T-})] \\
	     & \leq x + A\alpha(\mu) + e^{\discRate T /2} A_\phi \alpha(\mu) \\
	     & \leq x + A' \alpha(\mu).
    \end{align*}
    Hence, $\iterOp \phi (x, \mu) \leq x + e^{-\discRate T} A' \alpha(\mu)$, so $\iterOp \phi \in \mathcal{X}_\alpha$.

    It is left to show that $\iterOp$ is a strict contraction. By the properties of $\iterOp$ and the construction of $\alphaMetric$,
    \begin{align*}
		| \iterOp \phi (x, \mu) - \iterOp \psi (x,\mu) | & \leq e^{-\discRate T} \E_{x,\mu} [ \alphaMetric(\phi, \psi) \alpha(\mu_{T-}) ] \\
	    &\leq e^{-\discRate T} e^{\discRate T /2} \alphaMetric(\psi, \phi) \alpha(\mu) \\
	    &\leq e^{-\discRate T /2} \alphaMetric(\phi, \psi) \alpha(\mu).
    \end{align*}
    This implies that $\alphaMetric(\iterOp \phi, \iterOp \psi) \leq e^{-\discRate T /2} \alphaMetric(\phi, \psi)$, showing that $\iterOp$ is indeed a strict contraction.

    	The statement about the value function is proved exactly as in the proof of Theorem~\ref{t.verification}.
\end{proof}

\begin{remark}
	\label{rem:OUassumptions}
    Assumption \ref{ass:alpha} is satisfied by $C^\mu_t = \int_0^t \mu_s \dif{s} + \xsigma \xW_t$, where $\mu$ is the Ornstein--Uhlenbeck processes
    \[  \dif{\mu_t} = k(\mumean - \mu_t) \dif{t} + \musigma \dif{\muW}_t, \]
	where $k$, $\mumean$, and $\musigma$ are positive constants.
    By the time-scaled representation of Ornstein--Uhlenbeck processes, for $t \in [0,1]$ we have
    \begin{align*}
		\E_{x,\mu} [( \mu_t )^+ ] &= \E_{x,\mu} \left[ \left( \mu e^{-kt} + \mumean (1 - e^{-kt}) + \frac{\musigma}{\sqrt{2k}} e^{-kt} \muW_{e^{2kt}-1} \right)^+ \right] \\
		 &\leq \mu^+ + \mumean + \frac{\musigma}{\sqrt{2k}} \E_{x,\mu} \left[ \sup_{t \in [0,e^{2k}-1]} \muW_t \right] \\
	    &= \mu^+ + \mumean + \musigma \sqrt{\frac{e^{2k}-1}{k\pi}}.
    \end{align*}
	Hence, $\alpha(\mu) = \mu^+ + A$ satisfies the second condition of Assumption~\ref{ass:alpha} for any
	\[A \geq \frac{\mumean + \sigma \sqrt{\frac{e^{2k}-1}{k\pi}}}{e^{\discRate T / 2} - 1} \vee 1.\]
    However, this estimate is also sufficient for the first condition, since
	\[  \E_{x,\mu} \left[ \big(x + \int_0^1 \mu_t \dif{t} + \sigma \xW_t\big)^+ \right] \leq x^+ + \int_0^1 \E_{x,\mu} [( \mu_t )^+] \dif{t} + \sigma \sqrt{\frac{2}{\pi}}. \]
	We therefore conclude that the conditions of Assumption~\ref{ass:alpha} are satisfied for this choice of $C^\mu$ and $\mu$.
\end{remark}

\subsection{Numerical results}
\label{sec:results_2D}
Assuming the dynamic programming principle holds, it follows that the value function solves
\begin{equation}
	\label{eqn:continuousPDE2D}
	\min \left\{ -(\partial_t + \cashflowGenerator - \discRate) v(t,x,\mu), \quad v(t,x,\mu) - \sup_{i \geq 0} (v(t,x+i, \mu) - (1 + \pcost) i - \fcost ) \right\} = 0,
\end{equation}
with the boundary condition $v(t, 0, \mu) = 0$ in the viscosity sense, i.e.,
\begin{equation}
	\label{eqn:spaceBoundary2D}
	v(t, 0, \mu) = \max\{0, \,\, (\partial_t + \cashflowGenerator + 1 - \discRate) v(t,0,\mu), \,\, \sup_{i \geq 0} (v(t, i, \mu) - (1 + \pcost) i - \fcost\}.
\end{equation}
Like in the one-dimensional case, we will employ this PDE formulation for the numerical solution of the problem.

We solve the model for $\dif{C^\mu_t} = \mu_t \dif{t} + \xsigma \dif{\xW_t}$, where $\mu$ is an Ornstein--Uhlenbeck process.\footnote{Recall from Remark~\ref{rem:OUassumptions} that this class of processes satisfies Assumption~\ref{ass:alpha} required for numerical convergence.}
This model was explored for continuous dividend payments in \cite{reppen2017dividends}.
Also in this case, we opt for a semi-Lagrangian scheme, and for the same reason as in the one-dimensional model, we place the same boundary condition $v_x = e^{-\discRate (T-t)}$ on the upper boundary in the $x$-dimension.
In the $\mu$-dimension boundary conditions also have to be set.
For the sake of our calculations, we mirror the process $\mu$ at the boundary.
There are better choices, but we expect it to have a relatively small impact due to the Ornstein--Uhlenbeck process' strong inward drift at the boundary.\footnote{The results are consistent with disregarding the diffusion at the $\mu$-boundary. In fact, disregarding the diffusion at the lower boundary and mirroring at the upper seems to allow the smallest domain without impacting the free boundaries, i.e., retaining stability with respect to choosing a larger domain.}

\begin{figure}
	\centering
	\pgfplotsset{every axis/.append style={width=1.0\textwidth, colormap/viridis}}
	\begin{subfigure}{0.5\linewidth}
		\begin{tikzpicture}
			\coordinate (O) at (0,0);
			\begin{axis}[name=axis, ylabel=$x$, xlabel=$\mu$, enlargelimits=false, enlarge x limits={abs value=0.05,lower}, enlarge y limits={abs value=0.05,upper}, no markers, legend pos=south east, legend cell align=left,]
				\addplot[draw=none,name path=lower] table [x=mu, y=divLower, col sep=comma] {Figures/continuousnoissuance2D.csv};
				\addplot[draw=none,name path=upper] table [x=mu, y=divUpper, col sep=comma] {Figures/continuousnoissuance2D.csv};
				\addplot[fill=lightgray] fill between[of=lower and upper];
				\addplot[black] table [x=mu, y=lower, col sep=comma] {Figures/noissuance2D.csv} node[pos=0.0, inner sep=0pt] (muminLow) {};
				\addplot[black] table [x=mu, y=upper, col sep=comma] {Figures/noissuance2D.csv} node[pos=0.0, inner sep=0pt] (muminHigh) {};
				\draw let \p1 = (muminLow), \p2 = (muminHigh) in (\x1,\y1) -- (\x1, \y2);
			\end{axis}
		\end{tikzpicture}
	\end{subfigure}%
	\begin{subfigure}{0.5\linewidth}
		\begin{tikzpicture}
			\coordinate (O) at (0,0);
			\begin{axis}[name=axis, ylabel=$x$, xlabel=$\mu$, enlargelimits=false, enlarge x limits={abs value=0.05,lower}, enlarge y limits={abs value=0.05,upper}, no markers, legend pos=south east, legend cell align=left,]
				\addplot[draw=none,name path=lower] table [x=mu, y=lower, col sep=comma] {Figures/continuousissuance2D.csv};
				\addplot[draw=none,name path=upper] table [x=mu, y=upper, col sep=comma] {Figures/continuousissuance2D.csv};
				\addplot[fill=lightgray] fill between[of=lower and upper];
				\addplot[black] table [x=mu, y=lower, col sep=comma] {Figures/issuance2D.csv} node[pos=0.0, inner sep=0pt] (muminLow) {};
				\addplot[black] table [x=mu, y=upper, col sep=comma] {Figures/issuance2D.csv} node[pos=0.0, inner sep=0pt] (muminHigh) {};
				\draw let \p1 = (muminLow), \p2 = (muminHigh) in (\x1,\y1) -- (\x1, \y2);
			\end{axis}
		\end{tikzpicture}
	\end{subfigure}
	\caption{\label{fig:boundaries2D} \textbf{State space and the free boundaries of \eqref{eqn:continuousPDE2D}--\eqref{eqn:spaceBoundary2D} (black lines).}
		The left panel is without equity issuance and the right panel is with equity issuance.
		Between the two lines, it is optimal to not pay dividends, whereas outside it is.
		The gray area corresponds the same model, but allowing for continuous payments of dividends.
		The interpretation is the same, but the area between the lines is filled.
		The cash flow is given by $C^\mu_t = \int_0^t \mu_s \dif{s} + \xsigma \xW_t$, where $\dif{\mu_t} = k(\mumean - \mu_t) \dif{t} + \musigma \dif{\muW_t}$.
		Parameter values are $\fcost = 0.1$, $\pcost = 0.2$, $\rho = 0.05$, $T = 1$, $\xsigma = 0.1$, $k = 0.5$, $\mumean = 0.15$, $\musigma = 0.3$, and $\text{Cov}(\xW_t, \muW_t) = 0$.}
\end{figure}
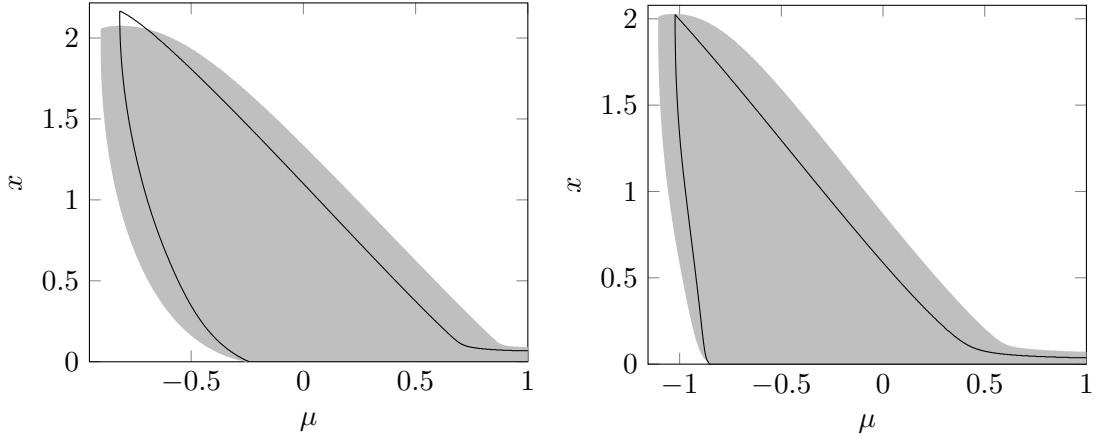

The dividend boundaries can be seen for models with and without equity issuance in Figure~\ref{fig:boundaries2D}.
As in \cite{reppen2017dividends}, we interpret the two lines constituting the dividend boundary in different ways.
The upper line has the same interpretation as the dividend barrier in the one-dimensional setting:
whenever the reserves are above it at the time of dividend payments, dividends are paid out such that the reserves move down to the line.
The lower boundary has a more subtle interpretation.
Mathematically seen, dividends are paid out whenever the reserves lie below this line.
Since the new state will still lie below the line, dividends must be paid until the reserves reach zero.
The interpretation of this is that the firm liquidates whenever the reserves dip below the line.
We will call these two lines the dividend boundary and the liquidation boundary.
For points to the left of these lines, the profitability is so low that liquidation is optimal regardless of reserves.

The general effect of dividend discretization is consistent in the two figures;
the dividend boundary moves downwards for most values of the profitability, with the exception of points close to where it meets the liquidation boundary.
Just like with constant profitability, we ascribe the lower dividend boundary to paying out profits in advance.
The liquidation boundary, on the other hand, moves upwards/inwards for all points.
This is likely due to the reduction in the prospective future value in the event of higher profitability, thus reducing today's value of not liquidating.
In both the continuous and discrete models, issuance only occurs at the boundary for the chosen parameter values.

\begin{figure}
\centering
\begin{tikzpicture}
	\pgfplotsset{every axis/.append style={axis on top,
			width=0.46\textwidth,
			enlargelimits=false,
			scaled ticks=false,
			tick label style={/pgf/number format/fixed},
			font=\small,
		}
	}
	\matrix{
		\begin{axis}[xlabel=$\mu$, ylabel=$x$]
			\addplot graphics[xmin=-0.9, ymin=0, xmax=1, ymax=2.2] {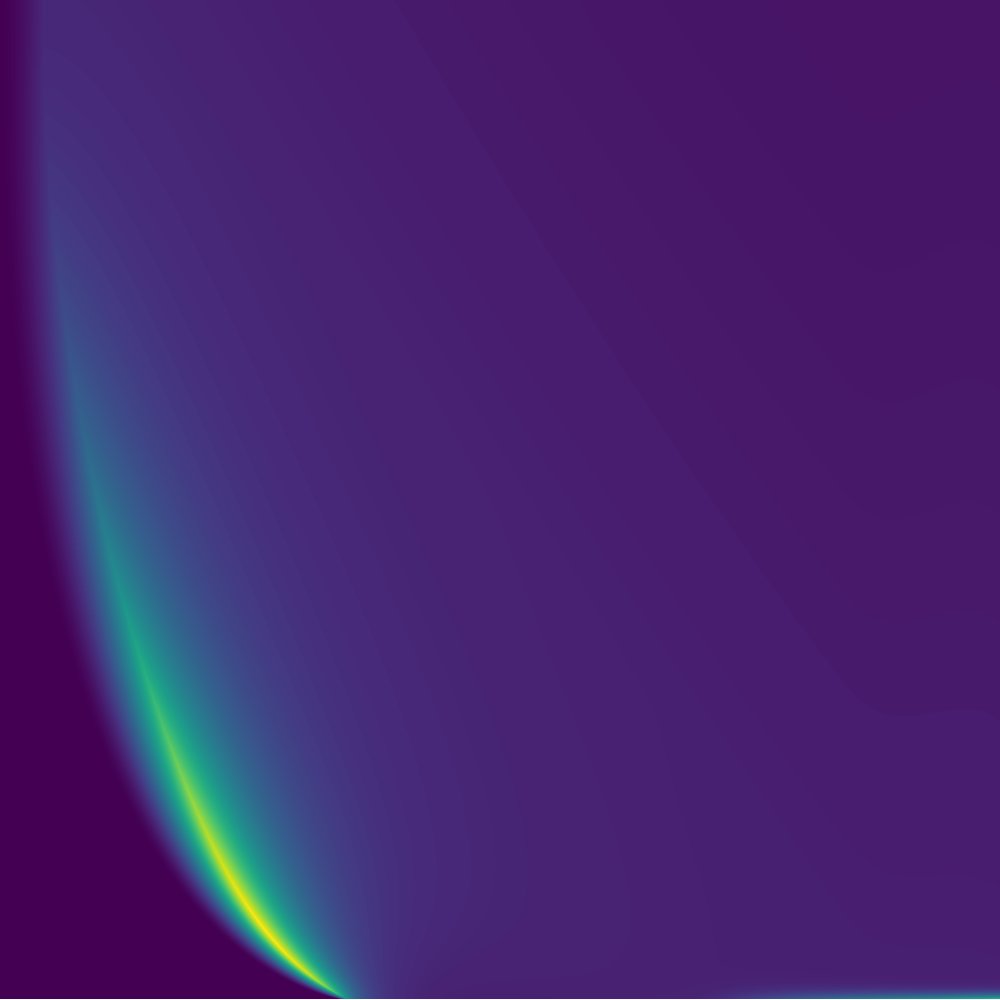};
			\addplot[white] table [x=mu, y=lower, col sep=comma] {Figures/noissuance2D.csv} node[pos=0.0, inner sep=0pt] (muminLow) {};
			\addplot[white] table [x=mu, y=upper, col sep=comma] {Figures/noissuance2D.csv} node[pos=0.0, inner sep=0pt] (muminHigh) {};
			\draw[white] let \p1 = (muminLow), \p2 = (muminHigh) in (\x1,\y1) -- (\x1, \y2);
		\end{axis}
		&
		\begin{axis}[xlabel=$\mu$, ylabel=$x$]
			\addplot graphics[xmin=-1.1, ymin=0, xmax=1, ymax=2.2] {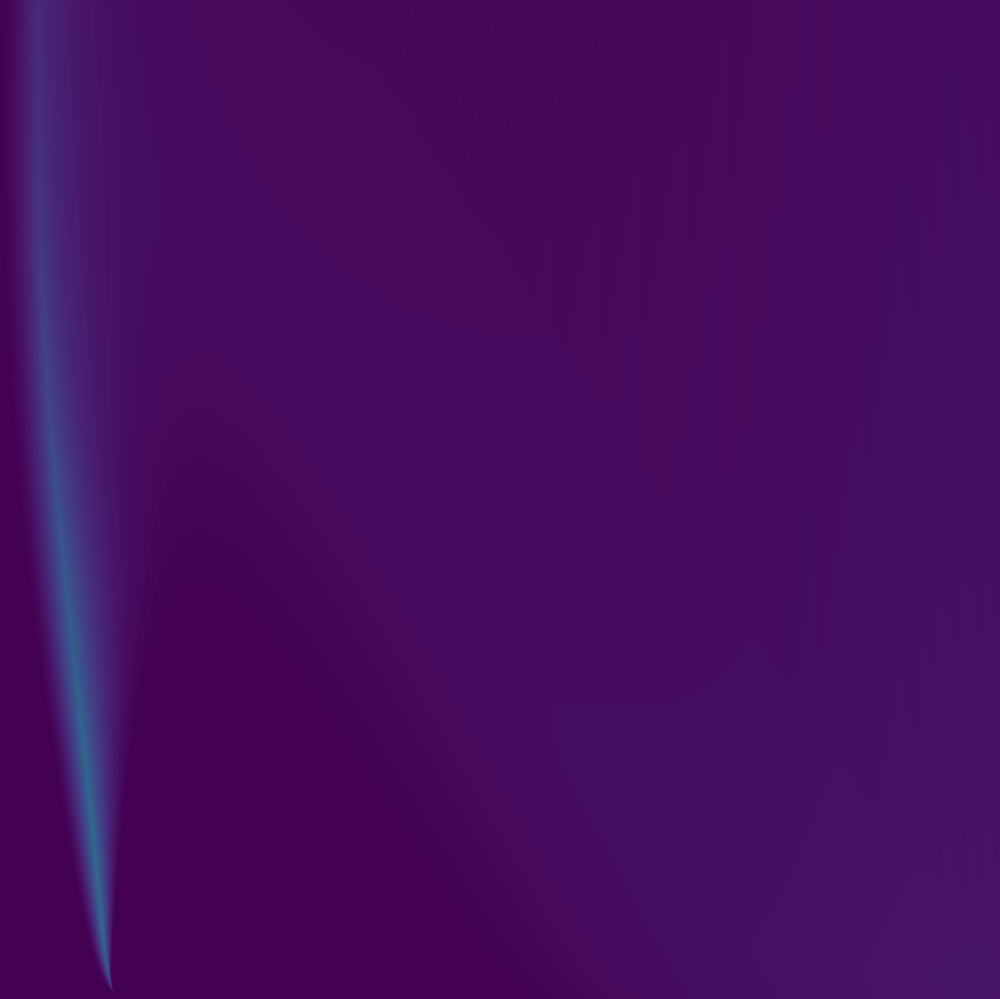};
			\addplot[white] table [x=mu, y=lower, col sep=comma] {Figures/issuance2D.csv} node[pos=0.0, inner sep=0pt] (muminLow) {};
			\addplot[white] table [x=mu, y=upper, col sep=comma] {Figures/issuance2D.csv} node[pos=0.0, inner sep=0pt] (muminHigh) {};
			\draw[white] let \p1 = (muminLow), \p2 = (muminHigh) in (\x1,\y1) -- (\x1, \y2);
		\end{axis}
		&
		\begin{axis}[axis equal image, xtick style={draw=none}, xticklabels={,,}, extra y ticks={23.616614205295264}, ylabel near ticks, yticklabel pos=right]
			\addplot graphics[xmin=0, ymin=0, xmax=2, ymax=23.616614205295264] {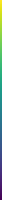};
		\end{axis}
		\\
	};
\end{tikzpicture}
\caption{\label{fig:loss2D}\textbf{Heatmap of the relative loss from discrete dividends relative to continuously paid dividends.}
	The left panel is without issuance and the right panel is with issuance.
	The scale is given in percent. Parameters are the same as in Figure~\ref{fig:boundaries2D}. The white curves constitute the dividend/liquidation boundaries for the discrete problem.}
\end{figure}
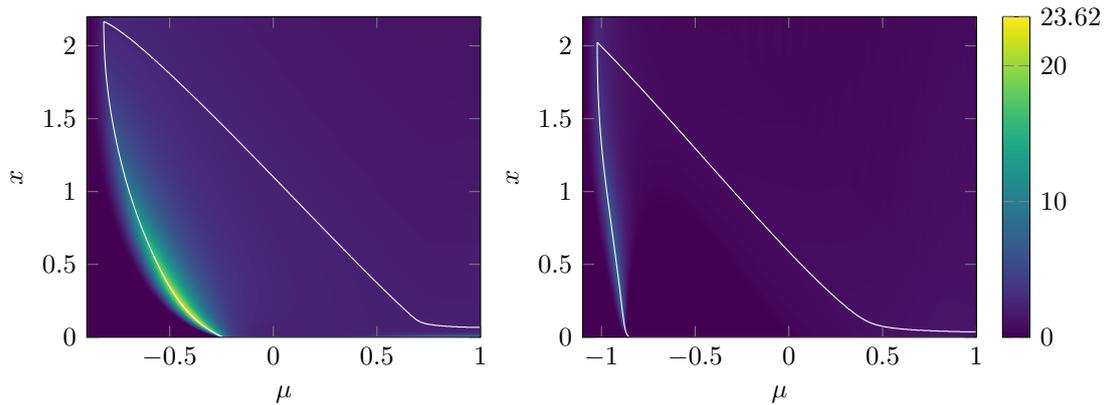

Figure~\ref{fig:loss2D} shows the relative loss of discrete dividends to continuous dividends for the various points in the state space.
The losses peak around the liquidation boundary for the discrete solution.
At these points, the losses are close to 25\% without issuance and a bit above 8\% with issuance.
For higher profitability and larger reserves, the losses soon dip below 3\% without issuance and 1\% with, decaying to 0 as $x$ increases.
In particular in the model with equity issuance, we see that the loss from discrete dividend payments is relatively small.
The average loss for all the points of the shown domain is less than 0.8\%.

\section{Concluding remarks}
\label{sec:conclusion}
For the one-dimensional dividend problem of Section~\ref{sec:dividends1D}, we find that the losses from dividend discretization are relatively low.
We have observed the same, relatively small losses also for other parameter choices, and believe that it extends to most reasonable choices in this one-dimensional setting.
In particular, for quarterly or more frequent dividends, the losses are especially small.
The overall small losses provide justification for using a continuous model as a substitute, if the goal is to find the value function/value for the cash flow.

On the other hand, the richer model presented in Section~\ref{sec:dividends2D} paints another picture.
For the parameters considered, we see that the total losses increase to almost 24\% in some parts of the state space.
This suggests that the dividend discretization can have a strong impact on the firm value.
Thus, whether the traditional continuous modelling is appropriate in a given setting is highly model-dependent.
In particular, the choice of dividends modelling has to be made on a case-by-case basis.

Finally, in the presented models, there is a pronounced shift in the optimal strategy.
This implies that using the optimal continuous dividend policy would induce further losses, as is illustrated in Figure~\ref{fig:noissuance1D}, possibly further affecting the performance loss.

\bibliographystyle{plainnat}
\bibliography{discretedividends}
\end{document}